\documentclass[reqno]{amsart}
\usepackage{amsmath,amsthm}
\usepackage{dsfont} 
\usepackage[english]{babel}
\usepackage{color}
\usepackage{amssymb}
\usepackage{mathrsfs}
\usepackage{graphicx}
\usepackage[font=footnotesize, labelfont=bf]{caption}
\usepackage{mathtools}
\usepackage[colorlinks=true, pdfstartview=FitV, linkcolor=blue, citecolor=blue, urlcolor=blue,pagebackref=false]{hyperref}
\usepackage{microtype}
\usepackage{enumitem}
\usepackage[normalem]{ulem}
\usepackage{soul}

\usepackage{float}

\newtheorem{theorem}{Theorem}[section]
\newtheorem{proposition}[theorem]{Proposition}
\newtheorem{lemma}[theorem]{Lemma}
\newtheorem{corollary}[theorem]{Corollary}

\theoremstyle{definition}
\newtheorem{remark}[theorem]{Remark}
\newtheorem{example}[theorem]{Example}
\newtheorem{definition}[theorem]{Definition}

\numberwithin{equation}{section}

\definecolor{myblue}{RGB}{0,0,128}

\newcommand{\LM}[1]{\hbox{\vrule width.2pt \vbox to#1pt{\vfill \hrule
width#1pt
height.2pt}}}
\newcommand{\LL}{{\mathchoice {\>\LM7\>}{\>\LM7\>}{\,\LM5\,}{\,\LM{3.35}\,}}}

\def\XXint#1#2#3{{\setbox0=\hbox{$#1{#2#3}{\int}$ }
		\vcenter{\hbox{$#2#3$ }}\kern-.57\wd0}}

\makeatletter
\newcommand{\oset}[3][0ex]{%
  \mathrel{\mathop{#3}\limits^{
    \vbox to#1{\kern-.75\ex@
    \hbox{$\scriptstyle#2$}\vss}}}}
\makeatother

\newcommand{\wcont}{\subset\subset}
\newcommand{\sm}{\setminus}
\newcommand{\dx}{\,\mathrm{d}x}

\newcommand{\e}{\varepsilon}
\newcommand{\dist}{{\rm{dist}}}

\newcommand{\w}{\omega}

\newcommand{\R}{\mathbb{R}}
\newcommand{\Z}{\mathbb{Z}}
\newcommand{\ZZ}{\mathcal{Z}}
\newcommand{\N}{\mathbb{N}}
\newcommand{\Q}{\mathbb{Q}}
\newcommand{\Sph}{\mathbb{S}}
\newcommand{\Sd}{\Sph^{d-1}}
\newcommand{\M}{\mathcal{M}}

\newcommand{\Hd}{\mathcal{H}^{d-1}}
\newcommand{\dHd}{\,\mathrm{d}\Hd}
\newcommand{\dH}{\,\mathrm{d}\mathcal{H}}

\newcommand{\A}{\mathcal{A}}
\newcommand{\m}{\mathbf{m}}
\newcommand{\Adm}{\mathrm{Adm}}
\newcommand{\F}{\mathcal{F}}
\newcommand{\I}{\mathcal{I}}
\renewcommand{\P}{\mathbb{P}}
\newcommand{\rb}{\partial^*\!}

\newcommand{\Eg}{E_1}

\newcommand{\Egw}{\Eg(\omega)}

\newcommand{\Hn}{H^\nu}
\newcommand{\Hnx}{\Hn(x_0)}
\newcommand{\Qn}{Q^\nu}
\newcommand{\Qnr}{\Qn_\rho}
\newcommand{\Qnxr}{\Qnr(x_0)}
\newcommand{\Br}{B_\rho}
\newcommand{\Brx}{\Br(x_0)}

\newcommand{\uzn}{u_0^{a,b,\nu}}

\newcommand{\ie}{\,{i.e.}, }

\newcommand{\defas}{:=}
\newcommand{\dsum}[2]{\sum_{\begin{smallmatrix}#1\\#2\end{smallmatrix}}}

\begin{document}
\author{Annika Bach}
\address[Annika Bach]{Dipartimento di Matematica ``Guido Castelnuovo'', Sapienza Universit\`{a} di Roma, Piazzale Aldo Moro 2,
00185 Roma, Italy}
\email{annika.bach@uniroma1.it}

\author{Matthias Ruf}
\address[Matthias Ruf]{Section de mathématiques, Ecole Polytechnique Fédérale de Lausanne, Station 8, 1015 Lausanne, Switzerland}
\email{matthias.ruf@epfl.ch}

\title[Stochastic homogenization of partitions]{Stochastic homogenization of functionals defined on finite partitions}

\begin{abstract}
We prove a stochastic homogenization result for integral functionals defined on finite partitions assuming the surface tension to be stationary and possibly ergodic. We also consider the convergence of boundary value problems when we impose a boundary value just on part of the boundary. As a consequence, we show that if the homogenized surface tension is isotropic, then one can obtain it by a multi-cell problem where Dirichlet boundary conditions are imposed only at the bottom and the top of the cube. We also show that this result fails in the general anisotropic case.
\end{abstract}

\maketitle
{\small
	\noindent
	
	\noindent
}	

%\tableofcontents

\section{Introduction}
Stochastic homogenization deals with the asymptotic description of models in a random environment that features oscillations on a micro-scale. As the micro-scale tends to zero, the goal is to find an effective model where the oscillations are replaced by a suitably averaged model (usually not the mean value, but a nonlinear average due to relaxation effects taking into account the micro-structure). Typical examples are PDEs with random coefficients or integral functionals with random integrands (see the monographs \cite{AKM_book,JKO_Hom} and references therein). In recent years, there has been a growing interest in the qualitative (i.e., without explicit error estimates) stochastic homogenization of variational functionals beyond the Sobolev setting: discrete-to-continuum limits for spin systems leading to interfaces \cite{ACR,BP,BP20}, Modica-Mortola-type functionals \cite{Marziani,Morfe}, free discontinuity problems \cite{CDMSZ19} and their approximations \cite{BCR19,BMZ21,R17}, $SBV$-functionals with linear growth \cite{CDMSZ22}, just to name a few. In \cite{BaRu22} we derived first quantitative estimates for the stochastic homogenization of functionals defined on finite partitions, i.e., functions of bounded variation taking values in a finite set. However, the qualitative homogenization result for this class of functionals has been derived only in the periodic setting in \cite{AmBrII}, the extension of this result to the random stationary setting instead was still missing. In this note we fill this gap. In particular, we prove a slightly more general version of \cite[Theorem 0.1]{BaRu22}, which we  stated there without including the proof. More precisely, given an open bounded Lipschitz set $D\subset\R^d$ and a parameter $\e>0$, we consider random heterogeneous interfacial energies of the form
\begin{equation}\label{intro:def-energy}
E_{\e}(\omega)(u)=\int_{D\cap S_u}g(\w,\tfrac{x}{\e},u^+,u^-,\nu_u)\dHd,
\end{equation}
where $u:D\to\M$ is a function of bounded variation taking values in a finite set $\M\subset\R^m$, $S_u$ is the set of approximate jump points of $u$, and $u^\pm$, $\nu_u$ are the traces of $u$ on $S_u$ and the measure theoretic normal to $S_u$, respectively. Moreover, $g:\Omega\times\R^d\times\M\times\M\times\Sd\to[0,+\infty)$ is a random surface tension depending on the outcome $\omega$ belonging to the sample space $\Omega$ of an underlying probability space $(\Omega,\F,\P)$. Assuming that $g$ is stationary and almost surely bounded from above and from below by deterministic constants (see Definition~\ref{defadmissible}) we show that for $\P$-almost all $\omega\in\Omega$ the functionals $E_\e(\omega)$ $\Gamma$-converge to a random homogeneous energy, given for $u\in BV(D;\M)$ by
\begin{equation}\label{intro:gamma-limit}
E_{\rm hom}(\omega)(u)=\int_{S\cap S_u}g_{\rm hom}(\omega,u^+,u^-,\nu_u)\dHd.
\end{equation}
For $(a,b,\nu)\in\M\times\M\times\Sd$ the integrand $g_{\rm hom}(\omega,a,b,\nu)$ is characterized via a multi-cell formula which is common in homogenization theory and involves the minimization of the energy at scale $\e=1$ on cubes $Q_t^\nu$ oriented with $\nu$ and with diverging side length $t$ under suitable boundary conditions. More precisely,
\begin{equation}\label{intro:ghom}
g_{\rm hom}(\omega,a,b,\nu)=\lim_{t\to+\infty}t^{1-d}\inf \big\{E_1(\omega)(u,Q_t^\nu)\colon u\in BV(Q_t^\nu;\M),\ u=u^{a,b,\nu}\ \text{near}\ \partial Q_t^\nu\big\}
\end{equation} 
where $u^{a,b,\nu}=a\mathds{1}_{\{x\cdot\nu>0\}}+b\mathds{1}_{\{x\cdot\nu<0\}}$ is the jump between the values $a$ and $b$ across the hyperplane orthogonal to $\nu$ and passing through the origin.

\smallskip
The present result in particular fills the gap between the qualitative $\Gamma$-convergence result \cite{AmBrII} in the periodic framework and the quantitative estimates derived in~\cite{BaRu22} on the level of the multi-cell formula defining $g_{\rm hom}$ in the stochastic framework. The result is slightly more general than the one announced in~\cite[Theorem 0.1]{BaRu22} as it allows for a general dependence of $g$ on the traces $u^\pm$ and not only on their difference $u^+-u^-$. This assumption was made in~\cite{BaRu22} for technical reasons that we will briefly comment on below while explaining the method of proof used to obtain the $\Gamma$-convergence result.

\smallskip
The proof of the result is split into a deterministic and a stochastic part, where the latter one deals with showing that the limit defining $g_{\rm hom}$ in~\eqref{intro:ghom} exists almost surely. This is done by rewriting the infimum problem in~\eqref{intro:ghom} in a suitable way as a subadditive stochastic process and by applying the subadditive ergodic theorem~\cite[Theorem 2.8]{AkKr}. To this end, it is necessary to show that the function that maps $\omega\in\Omega$ to the infimum value in~\eqref{intro:ghom} is measurable with respect to the $\sigma$-algebra $\F$. In the case where $g$ depends on $u^\pm$ only through their difference $u^+-u^-$, this follows from~\cite[Proposition A.1]{CDMSZ19}. In the present paper we prove the measurability by reducing to the case, where the elements of $\M$ are the standard orthonormal basis, so that the mapping $(u^+-u^-)\mapsto (u^+,u^-)$ is well defined. Note that the possibility of associating a subadditive stochastic process to the infimum problem in~\eqref{intro:ghom} is also related to the fact that the cubes $Q_t^\nu$ are concentric. To prove the full homogenization result, an additional probabilistic  step is required, where one shows that $g_{\rm hom}(\omega,a,b,\nu)$ coincides with the limit
\begin{equation}\label{intro:limitx0}
\lim_{t\to+\infty}t^{1-d}\inf \big\{E_1(\omega)(u,Q_t^\nu(t x_0)\colon u\in BV(Q_t^\nu(t x_0);\M),\ u=u_{t x_0}^{a,b,\nu}\ \text{near}\ \partial Q_t^\nu(t x_0)\big\}
\end{equation}
for every $x_0\in\R^d$, where the cubes are now centered at $t x_0$ and $u_{t x_0}^{a,b,\nu}$ correspondingly jumps across the hyperplane orthogonal to $\nu$ and passing through $t x_0$.

Based on the existence and equality of the two limits in~\eqref{intro:ghom} and~\eqref{intro:limitx0}, the $\Gamma$-convergence result can be obtained with completely deterministic arguments by applying the blow-up method together with the fundamental estimate~\cite[Lemma 4.4]{AmBrI} to establish the liminf-inequality and the density result in~\cite{BCG} to reduce the limsup-inequality to polyhedral partitions. We emphasize that the explicit construction of the recovery sequence for polyhedral partitions is still quite technical due to the presence of more than two phases. While in \cite{BaRu22} we indicated a proof of the homogenization result using integral representation techniques \cite[Theorem 3]{BFLM}, thus following closely the strategy of \cite{ACR,BrCi,BCR,CDMSZ19}, here we decided to take the more direct approach as it has the chance to be applicable also in the setting of degenerate growth conditions; cf. \cite{DG,RR23,RZ23} for recent results in the Sobolev setting. We plan to include this analysis in future work.

\smallskip
Eventually, we also prove a $\Gamma$-convergence result for the functionals $E_\e(\omega)$ in presence of Dirichlet boundary conditions on part of $\partial D$. As a consequence of this we are able to show that in the isotropic case, i.e., when $g_{\rm hom}$ does not depend on $\nu$, the integrand $g_{\rm hom}(\omega,a,b)$ can be equivalently characterized by a limit of minimization problems as in~\eqref{intro:ghom}, where now the boundary conditions are only prescribed on the upper and lower facets $\{x\cdot\nu=\pm t/2\}$ of $Q_t^\nu$ (see Corollary~\ref{cor:limit-alternativ-process}). We also show that this is in general not true in the anisotropic case (see Example~\ref{ex:failure_top_bottom}).

%Indeed, to prove the lower-bound inequality we apply the blow-up method which in combination with the fundamental estimate~\cite[Lemma 4.4]{AmBrI} essentially allows us to 
%
%\vspace*{2cm}
%{\BBB The aysmptotic lower bound for $E_\e(\omega)(u_\e,D)$ along sequences $u_\e$ converging in $L^1$ to some $u\in BV(D;\M)$ is obtained via the blow-up method. Combining the latter one with the fundamental estimate~\cite[Lemma 4.4]{AmBrI} allows one to reduce the problem to showing that} the energy $E_\e(\omega)(u_\e,Q_\rho^\nu(x_0))$ on small balls $Q_\rho^\nu(x_0)$ centred at a jump point $x_0$ of $u$ and oriented with $\nu=\nu_u(x_0)$ along sequences $u_\e$ coinciding with the pure jump between the values $u^+(x_0)$ and $u^-(x_0)$ near $\partial Q_\rho^\nu(x_0)$ is bounded from below by $g_{\rm hom}(\omega,u^+(x_0),u^-(x_0),\nu_u(x_0))$ when $\e\to 0$ and $\rho\to 0$. Let us suppose for a moment that $x_0=0$; then the desired lower bound would follow via a change of variables and setting $(a,b,\nu)=(u^+(x_0),u^-(x_0),\nu_u(x_0))$ in~\eqref{intro:ghom} once we know that the limit in~\eqref{intro:ghom} exists. 
%
%
%to ensure the measurability with respect to $\F$ of a suitable stochastic process defining $g_{\rm hom}$ (see~\eqref{def:subadprocess}). 
%
%\RRR Note for me: continue with intro [[formulate problem, state result, comment on proof $\to$ measurability and convergence of Dirichlet problems with partial boundary]]\EEE

\section{Setting of the problem and statement of the main results}\label{s.setting-result}
\subsection{General notation}
We first introduce some notation that will be used in this paper. Given a measurable set $A\subset\R^d$ we denote by $|A|$ its $d$-dimensional Lebesgue measure, and by $\mathcal{H}^{k}(A)$ its $k$-dimensional Hausdorff measure. Moreover, $\mathscr{A}$ is the family of open, bounded subsets of $\R^d$ with Lipschitz boundary.
For $x,y\in\R^d$ we denote by $x\cdot y$ the scalar product between $x$ and $y$, $|x|$ is the Euclidean norm, and $\Brx$ denotes the open ball with radius $\rho>0$ centered at $x_0\in\R^d$. If $x_0=0$ we simply write $\Br$.
Given $x_0\in\R^d$ and $\nu\in\Sph^{d-1}$, we let $\Hnx$ be the hyperplane orthogonal to $\nu$ and passing through $x_0$ and for every $(a,b)\in\R^m\times\R^m$, the piecewise constant function taking values $a,b$ and jumping across $\Hnx$ is denoted by $u_{x_0}^{a,b,\nu}:\R^{d}\to\R^{m}$,\ie
\begin{equation}\label{eq:purejump}
u_{x_0}^{a,b,\nu}(x):=\begin{cases} a &\mbox{if $(x-x_0)\cdot\nu>0$,}
\\
b &\mbox{otherwise.}
\end{cases}
\end{equation}
If $x_0=0$ we write $\Hn$ in place of $H^\nu(0)$.
Let $\{e_1,\ldots,e_d\}$ be the standard basis of $\R^d$. Then $O_{\nu}$ is the orthogonal matrix induced by the linear mapping
\begin{equation}\label{eq:matrix}
x\mapsto \begin{cases}
\displaystyle{2\frac{(x\cdot\nu+e_d)}{|\nu+e_d|^2}(\nu+e_d)-x} &\mbox{if $\nu\in \mathbb{S}^{d-1}\setminus\{-e_d\}$,}
\\
-x &\mbox{otherwise.}
\end{cases}
\end{equation}
In this way, $O_\nu e_d=\nu$ and the set $\{\nu_j\defas O_\nu e_j\colon j=1,\ldots, d-1\}$ is an orthonormal basis for $\Hn$. 
Setting $\nu_d=\nu$, we define the cube $\Qn$ as
\begin{equation}\label{eq:defcube}
\Qn=\left\{x\in\R^d\colon |x\cdot\nu_j| <1/2\text{ for }j=1,\ldots,d\right\},
\end{equation}
and we set $\Qnxr=x_0+\rho \Qn$. 
%We further denote by $\Hnx$ the hyperplane orthogonal to $\nu$ and passing through $x_0$. If $x_0=0$ we write $\Hn$. 
%
%
%
%For $x_0\in \R^d$, $\nu\in \mathbb{S}^{d-1}$ and $\zeta\in\R^m$ define the function $\uxzn:\R^{d}\to\R^{m}$ as
%\begin{equation}\label{eq:purejump}
%\uxzn(x):=\begin{cases} \zeta &\mbox{if $\langle x-x_0,\nu\rangle >0$,}
%\\
%0 &\mbox{otherwise.}
%\end{cases}
%\end{equation}
%If $x_0=0$ we simply write $\uzn$ in place of $u_{0}^{\zeta,\nu}$.

Finally, the letter $C$ stands for a generic positive constant that may change every time it appears.
\subsection{BV-functions and finite partitions}\label{s.BV}
In this section we recall basic facts about functions of bounded variation. For more details we refer to the monograph \cite{AFP}. Let $U\subset\R^d$ be an open set. A function $u\in L^1(U;\R^m)$ is a function of bounded variation if its distributional derivative $D u$ is given by a finite matrix-valued Radon measure on $U$. In that case, we write $u\in BV(U;\R^m)$. The space $BV_{{\rm loc}}(U;\mathbb{R}^m)$ is defined as usual. The space $BV(U;\R^m)$ becomes a Banach space when endowed with the norm $\|u\|_{BV(U)}=\|u\|_{L^1(U)}+|D u|(U)$, where $|D u|$ denotes the total variation measure of $D u$. When $U\in\mathscr{A}$, then $BV(U;\R^m)$ is compactly embedded in $L^1(U;\R^m)$. 

\smallskip
Let us state some fine properties of $BV$-functions. To this end, we need some definitions. A function $u\in L^1(U;\mathbb{R}^m)$ is said to have an approximate limit at $x\in U$ whenever there exists $z\in\mathbb{R}^m$ such that
\begin{equation*}
	\lim_{\rho\to 0}\frac{1}{\rho^d}\int_{B_{\rho}(x)}|u(y)-z|\,\mathrm{d}y=0\,.
\end{equation*}
Next we introduce so-called approximate jump points. Given $x\in U$ and $\nu\in \mathbb{S}^{d-1}$ we set
\begin{equation}\label{def:Qpm}
	Q_\rho^{\nu,\pm}(x)\defas\{y\in Q_\rho^\nu(x):\;\pm (y-x) \cdot \nu >0\}\,.
\end{equation}
We say that $x\in U$ is an approximate jump point of $u$ if there exist $a\neq b\in\mathbb{R}^n$ and $\nu\in \mathbb{S}^{d-1}$ such that
\begin{equation}\label{def:jump-point}
	\lim_{\rho\to 0}\frac{1}{\rho^d}\int_{Q_\rho^{\nu,+}(x)}|u(y)-a|\,\mathrm{d}y=\lim_{\rho\to 0}\frac{1}{\rho^d}\int_{Q_\rho^{\nu,-}(x)}|u(y)-b|\,\mathrm{d}y=0 \, .
\end{equation}
The triplet $(a,b,\nu)$ for such $x\in U$ is determined uniquely up to the change to $(b,a,-\nu)$ and denoted by $(u^+(x),u^-(x),\nu_u(x))$. We let $S_u$ be the set of approximate jump points of $u$. The triplet $(u^+,u^-,\nu_u)$ can be chosen as a Borel function on the countably $\mathcal{H}^{d-1}$-rectifiable Borel set $S_u$. 

\smallskip
We will also make use of the slicing properties of $BV$-functions (cf. \cite[Section 3.11]{AFP}). Let $\nu \in \mathbb{S}^{d-1}$. For any Borel set $E\subset U$ and every $y \in H^{\nu}$, the section of $E$ corresponding to $y$ is the set $E^\nu_y := \{ t \in \R \ : \ y + t \nu \in E \}$. Accordingly, for any function $u \colon U \to \R^m$, the function $u^\nu_y \colon U^\nu_y \to \R^m$ is defined by $u^\nu_y(t) := u(y + t \nu)$. For $u \in L^1(U; \R^m)$ we have $u \in BV(U; \R^m)$ if and only if for every $\nu \in \mathbb{S}^{d-1}$ we have $u^\nu_y \in BV(U^\nu_y;\R^m)$ for $\mathcal{H}^{d-1}$-a.e.\ $y \in H^\nu$ and 
\[
\int_{H^\nu}{|D u^\nu_y|(U^\nu_y)}{ \dHd(y)} < +\infty \, .
\]
It is possible to reconstruct the distributional gradient $D u$ from the gradients of the slices $D u^\nu_y$ through the formula
\[
D u \, \nu(B) = \int_{H^\nu}{D u^\nu_y(B^\nu_y)}{\dHd(z)} \, ,
\]
for every Borel set $B \subset U$. Moreover, $S_{u^\nu_y} = (S_u)^\nu_y$ for $\mathcal{H}^{d-1}$-a.e.\ $y \in H^\nu$ and $(u^\nu_y)^{\pm}(t) = (u^\pm)^\nu_y(t)$  ($ = (u^\mp)^\nu_y(t)$, respectively) for every $t \in (S_u)^\nu_y$ if $\nu \cdot \nu_u(y+t\xi) > 0$ (if $\nu \cdot \nu_u(y+t\xi) < 0$, respectively).

\medskip
The relevant function space in this paper is the set of finite partitions,\ie the space of functions of bounded variation taking only finitely many values. Given a finite set $\M=\{a_1,\ldots,a_K\}\subset\R^m$, we set $BV(U;\M)\defas\{u\in BV(U;\R^m)\colon u(x)\in\M\text{ a.e.\! in }U\}$. In this case, the measure $Du(B)$ of a Borel set $B\subset U$ can be represented as 
\begin{equation*}
Du(B)=\int_{B\cap S_u}(u^+(x)-u^-(x))\otimes\nu_u(x)\dHd.	
\end{equation*}
Moreover, it follows from the coarea formula that each level set $A_i\defas\{u=a_i\}$ has finite perimeter in $U$, \ie $\mathds{1}_{A_i}\in BV(U;\{0,1\})$. Eventually, we recall that the reduced boundary $\rb A_i$ of $A_i$ is the set of all points $x\in{\rm supp}|D_{\mathds{1}_{A_i}}|$ for which the limit $\nu_{A_i}(x)\defas\lim_{\rho\to 0}\frac{D_{\mathds{1}_{A_i}}(B_\rho(x))}{|D_{\mathds{1}_{A_i}}|(B_\rho(x))}$ exists and satisfies $|\nu_{A_i}(x)|=1$. 
It holds that
	\begin{equation*}
	U\cap S_u=U\cap \bigcup_{i=1}^K\bigcup_{j\neq i}\rb A_i\cap\rb A_j
	\end{equation*}
up to an $\Hd$-negligible set and $(u^+,u^-,\nu_u)=(a_i,a_j,\nu_{A_i})$ $\Hd$-a.e.\ on $\rb A_i\cap\rb A_j$ (see~\cite[Proposition 5.9]{AFP}).
\subsection{Ergodic theory}
In this section, we recall some basic notions from probability theory. Throughout this paper $(\Omega,\F,\P)$ denotes a complete probability space.
We start by defining measure-preserving group actions.
\begin{definition}[Measure-preserving group action]\label{def:group-action} Let $k\in \N$, $k\geq 1$. A \emph{discrete, measure-preserving, additive group action} on $(\Omega,\F,\P)$  is a family $\{\tau_z\}_{z\in\Z^k}$ of mappings $\tau_z:\Omega\to\Omega$ satisfying the following properties:
	\begin{enumerate}[label=(\arabic*)]
		\item\label{meas} (measurability) $\tau_z$ is $\F$-measurable for every $z\in\Z^k$;
		\item\label{inv} (invariance) $\P(\tau_z A)=\P(A)$, for every $A\in\F$ and every $z\in\Z^k$;
		\item\label{group} (group property) $\tau_0=\rm id_\Omega$ and $\tau_{z_1+z_2}=\tau_{z_2}\circ\tau_{z_1}$ for every $z_1,z_2\in\Z^k$.
	\end{enumerate}
	If, in addition, $\{\tau_z\}_{z\in\Z^k}$ satisfies the implication
	\begin{equation*}
		\mathbb{P}(\tau_zA\Delta A)=0\quad\forall\, z\in\Z^k\implies \mathbb{P}(A)\in\{0,1\},
	\end{equation*}
	then it is called ergodic.
\end{definition}
The following is known as Birkhoff's additive ergodic theorem:
\begin{theorem}\label{thm:Birkhoff}
Let $f\in L^1(\Omega)$ and $\{\tau_z\}_{z\in\Z^k}$ be a discrete, measure-preserving group action. Then for $\mathbb{P}$-a.e. $\w\in\Omega$ and every $z\in\Z^k$ we have	
\begin{equation*}
	\sum_{i\in\N}f(\tau_{iz}\w)=\mathbb{E}[f|\mathcal{J}_{z}](\w),
\end{equation*}
where $\mathcal{J}_z$ is the $\sigma$-algebra of $\tau_{z}$-invariant sets, i.e., $\mathcal{J}_z=\{A\in\mathcal{F}:\,\mathbb{P}(\tau_z A\Delta A)=0\}$, and $\mathbb{E}[f|\mathcal{J}_z]$ denotes the conditional expectation of $f$ with respect to $\mathcal{J}_z$, i.e., the (up to null sets uniquely defined) $\mathcal{J}_z$-measurable function $h:\Omega\to\R$ such that $\mathbb{E}[\mathds{1}_Af]=\mathbb{E}[\mathds{1}_A h]$ for all $A\in\mathcal{J}_z$.
\end{theorem}
We further recall here the basics of subadditive stochastic processes tailored to our analysis on $d-1$-dimensional subspaces:

For every $p=(p_1,\dots,p_{d-1}),\,q=(q_1,\dots,q_{d-1})\in\R^{d-1}$ with $p_i<q_i$ for all $i\in\{1,\ldots,d-1\}$ we consider the $(d-1)$-dimensional half-open intervals 
\begin{equation*}
	[p,q):=\{x\in\R^{d-1}\colon p_i\le x_i<q_i\,\,\text{for}\,\,i=1,\ldots,d-1\}
\end{equation*}
and we set
\begin{equation*}
	\I:=\{[p,q)\colon p,q\in\R^{d-1}\,,\, p_i<q_i\,\,\text{for}\,\,i=1,\ldots,d-1\}\,.
\end{equation*}
\begin{definition}[Subadditive process]\label{sub_proc}
	A (bounded) subadditive process with respect to a discrete, measure-preserving, additive group action $\{\tau_z\}_{z\in\Z^{d-1}}$ is a function $\mu\colon\I\times\Omega\to\R$ satisfying the following properties:
	\begin{enumerate}[label=(\arabic*)]
		\item\label{subad:meas} (measurability) for every $I\in\I$ the function $\omega\mapsto\mu(I,\omega)$ is $\F$-measurable;
		\item\label{subad:cov} (stationarity) for a.e. $\omega\in\Omega$, $I\in\I$, and $z\in\Z^{d-1}$ we have $\mu(I+z,\omega)=\mu(I,\tau_z(\omega))$;
		\item\label{subad:sub} (subadditivity) for a.e. $\w\in\Omega$, every $I\in\I$ and for every finite partition $(I^i)_{i=1}^k$ of $I$, we have
		\begin{equation*}
			\mu(I,\omega)\leq\sum_{i=1}^k\mu(I_i,\omega)\quad\text{for every}\,\,\omega\in\Omega\,;
		\end{equation*}
		\item\label{subad:bound} (boundedness) there exists $M>0$ such that $0\leq\mu(I,\omega)\le M\mathcal{L}^{d-1}(I)$ for a.e. $\omega\in\Omega$ and $I\in\I$.
	\end{enumerate}
\end{definition}
We will use the following version of the subadditive ergodic theorem which is a special case of \cite[Theorem 2.8]{AkKr}.
\begin{theorem}[Subadditive ergodic theorem]\label{thm:subadditive}
Let $\mu:\mathcal{I}\times\Omega\to\R$ be a subadditive stochastic process. Then there exists a $\mathcal{F}$-measurable function $\phi:\Omega\to\R$ such that $\mathbb{P}$-almost surely for every cube $Q$
\begin{equation*}
\lim_{t\to +\infty}\frac{\mu(tQ,\w)}{\mathcal{L}^{d-1}(tQ)}=\phi(\w).
\end{equation*}	
\end{theorem}
\begin{remark}
	Strictly speaking, \cite[Theorem 2.8]{AkKr} only allows us to consider rational sequences $t\to +\infty$ and cubes with rational vertices. However, the pointwise bound in Definition \ref{def:subadprocess} (4) and pointwise subadditivity  (except for a null set not depending on the cubes!) help to extend the convergence to the more general setting by an inner and outer approximation argument.
\end{remark}

\subsection{Boundedness and probabilistic assumptions}
In this subsection we give the precise assumptions we make on the random integrand. Throughout this note we fix a finite set $\M\subset\R^m$ and a parameter $c\geq 1$. We denote by $\A_c$ the class of all Borel measurable functions $g:\R^d\times \M\times\M\times \mathbb{S}^{d-1}\to [0,+\infty)$ satisfying
\begin{equation}\label{cond:Ac}
\frac{1}{c}\leq g(x,a,b,\nu)\leq c,\quad \text{and}\quad g(x,a,b,\nu)=g(x,b,a,-\nu)\,,  
\end{equation}
for every $(x,a,b,\nu)\in\R^d\times\M\times \M\times \mathbb{S}^{d-1}$.
%To any $g\in\A_c$ and $D\subset\R^d$ open with Lipschitz boundary we associate a functional $\Eg(\cdot,D)$ defined on partitions by setting $\Eg(\cdot,D):L^1_{\rm loc}(\R^d;\R^m)\to[0,+\infty]$,
%%
%\begin{equation}\label{def:energy-deterministic}
%\Eg(u,D)\defas
%\begin{cases}
%\displaystyle \int_{D\cap S_u} g(x,u^+-u^-,\nu_u)\dHd &\text{if }u\in BV(D;\M),\\
%+\infty &\text{otherwise in }L^1_{\rm loc}(\R^d;\R^m).
%\end{cases}
%\end{equation}
%Here $\mathcal{M}\subset\R^m$ is a finite set that we fix throughout this paper. Note that $\Eg$ is well-defined for $g\in\A_c$ thanks to the second condition in~\eqref{cond:Ac} and the fact that the triple $(u^+(x),u^-(x),\nu_u(x))$ is uniquely defined up to a permutation in $(u^+(x),u^-(x))$ and a simultaneous change of sign in $\nu_u(x)$.
%%

We next introduce the random setting considered in the notion of an admissible random surface integrand.
\begin{definition}[Admissible surface tensions]\label{defadmissible} 
We say that a function $g:\Omega\times\R^d\times\M\times\M\times\mathbb{S}^{d-1}\to [0,+\infty)$ is an admissible random surface tension, if it is jointly measurable and if for almost every $\omega\in\Omega$ the function $g(\omega)\colon\R^d\times\M\times\M\times\Sd\to[0,+\infty)$, $(x,a,b,\nu)\mapsto g(\omega,x,a,b,\nu)$ belongs to $\mathcal{A}_c$.
It is said to be $\Z^d$-stationary if there exists a discrete, measure-preserving, additive group action $\{\tau_{z}\}_{z\in\Z^d}$ such that for a.e. $\omega\in\Omega$ and for all $z\in\Z^d$ it holds that
\begin{equation*}
g(\tau_z\w,x,a,b,\nu)=g(\w,x+z,a,b,\nu)\quad\quad\forall\, (x,a,b,\nu)\in \R^d\times\M\times\M\times\mathbb{S}^{d-1}.
\end{equation*}
If, in addition, the group action $\{\tau_z\}_{z\in\Z^d}$ is ergodic, then $g$ is called ergodic.
\end{definition}
For any admissible random surface tension $g$, any $\omega\in\Omega$  with $g(\omega)\in\A_c$, any $\e>0$, and $D\in\mathscr{A}$ we consider the random functionals  $E_\e(\omega)(\cdot,D):L^1(D;\R^m)\to[0,+\infty]$ given by 
	\begin{equation}\label{def:energy}
	E_\e(\omega)(u,D)\defas
	\begin{cases}
	\displaystyle \int_{D\cap S_u} g(\omega,\tfrac{x}{\e},u^+,u^-,\nu_u)\dHd &\text{if }u\in BV(D;\M),\\
	+\infty &\text{otherwise on }L^1(D;\R^m).
	\end{cases}
	\end{equation}
We also consider constrained versions of $E_\e(\omega)$ by prescribing Dirichlet boundary conditions on part of the boundary. 
To be more precise, fix a relatively open subset $\Gamma\subset\partial D$ such that its relative boundary $\partial_{\rm rel}\Gamma$ in $\partial D$ is $\mathcal{H}^{d-1}$-negligible and a function $u_0\in BV_{\rm loc}(\R^d;\M)$ such that $\mathcal{H}^{d-1}(S_{u_0}\cap\partial D)=0$. In order to study the Dirichlet problem with constraint $u=u_0$ on $\Gamma$, we introduce the space
\begin{equation*}
	BV_{\Gamma,u_0}(D;\M)=\{u\in BV(D;\M):\,u|_{\partial D}=u_0|_{\partial D}\;\mathcal{H}^{d-1}\text{-a.e. on }\Gamma\},
\end{equation*}
where the symbol $|_{\partial D}$ denotes the trace on $\partial D$. We further define the constrained functional $E_{\e,\Gamma,u_0}(\w)$ by
\begin{equation*}
	E_{\e,\Gamma,u_0}(\w)(u,D)=\begin{cases}
		\displaystyle \int_{D\cap S_u}g(\w,\tfrac{x}{\e},u^+,u^-,\nu_u)\,\mathrm{d}\mathcal{H}^{d-1} &\mbox{if $u\in BV_{\Gamma,u_0}(D;\M)$,}
		\\
		+\infty &\mbox{otherwise on $L^1(D;\R^m)$.}
	\end{cases}
\end{equation*}

Throughout this note, $\e>0$ will vary in a strictly decreasing family of positive parameters converging to zero. 
%In order to study the stochastic homogenization of $E_\e$ and $E_{\e,\Gamma,u_0}$ as $\e\to 0$ we now introduce a further probabilistic assumption regarding the spatial stationarity of the integrand (cf. Definition \ref{def:group-action})
%
%\begin{assumption}\label{ass:1}
%The admissible random surface tension is $\Z^d$-stationary,\ie there exists a discrete measure-preserving group action $\{\tau_{z}\}_{z\in\Z^d}$ such that for all $\omega\in\Omega$ and for all $z\in\Z^d$ it holds that
%\begin{equation*}
%g(\tau_z\w,x,a,b,\nu)=g(\w,x+z,a,b,\nu)\quad\quad\forall\, (x,a,b,\nu)\in \R^d\times\M\times\M\times\mathbb{S}^{d-1}.
%\end{equation*}
%It is called ergodic, if it is stationary and the group action $\{\tau_z\}_{z\in\Z^d}$ is ergodic. 
%%We refer to Assumption~1(E) if ergodicity holds.
%\end{assumption}
\subsection{Statement of the main results}
Below we state the main results of this note.
\begin{theorem}\label{thm:main}
Let $g$ be an admissible random surface tension in the sense of Definition~\ref{defadmissible}, let $E_\e$ be as in~\eqref{def:energy}, and suppose that $g$ is $\Z^d$-stationary. There exists an event $\Omega'\subset\Omega$ with $\P(\Omega')=1$ such that for every $\omega\in\Omega'$ and every $D\in\mathscr{A}$ the functionals $E_\e(\omega)(\cdot,D)$ $\Gamma$-converge with respect to the strong convergence in  $L^1(D;\R^m)$ to the functional $E_{\rm hom}(\omega)(\cdot,D):L^1(D;\R^m)\to [0,+\infty]$ defined as
	\begin{equation}\label{intro:Ehom}
	E_{\rm hom}(\omega)(u,D)\defas
	\begin{cases}
	\displaystyle\int_{D\cap S_u}g_{\rm hom}(\w,u^+,u^-,\nu_u)\,\mathrm{d}\mathcal{H}^{d-1} &\text{if}\ u\in BV(D;\M),\\
	+\infty &\text{otherwise in}\ L^1(D;\R^m),
	\end{cases}
	\end{equation}
%where $g_{\rm hom}:\Omega\times\M\times\M\times\Sd\to[0,+\infty)$ is an admissible random surface tension characterised for $\omega\in\Omega'$ by the multi-cell formula
where $g_{\rm hom}(\omega,\cdot,\cdot,\cdot):\M\times\M\times\Sd\to[0,+\infty)$ is given by the multi-cell formula 
	\begin{equation}\label{def:ghom}
	g_{\rm hom}(\omega,a,b,\nu)\defas\lim_{t\to+\infty}\frac{1}{t^{d-1}}\inf\big\{E_1(\omega)(u,Q_t^\nu)\colon u\in BV(Q_t^\nu;\M),\ u=u_0^{a,b,\nu}\ \text{near}\ \partial Q_t^\nu\big\}\,.
	\end{equation}
In particular, the limit in~\eqref{def:ghom} exists for every $\omega\in\Omega'$ and every $(a,b,\nu)\in\M\times\M\times\Sd$. Moreover, for every $\omega\in\Omega'$ the mapping $(a,b,\nu)\mapsto g_{\rm hom}(\omega,a,b,\nu)$ is continuous.
Eventually, if $g$ is ergodic, then $g_{\rm hom}$ is independent of $\omega$.
\end{theorem}
%We also consider Dirichlet boundary conditions on a part of the boundary. To be more precise, fix a relatively open subset $\Gamma\subset\partial D$ such that its relative boundary $\partial_{\rm rel}\Gamma$ in $\partial D$ is $\mathcal{H}^{d-1}$-negligible and a function $u_0\in BV_{\rm loc}(\R^d;\M)$ such that $\mathcal{H}^{d-1}(S_{u_0}\cap\partial D)=0$. In order to study the Dirichlet problem with constraint $u=u_0$ on $\Gamma$, we introduce the space
%\begin{equation*}
%	BV_{\Gamma,u_0}(D;\M)=\{u\in BV(D;\M):\,u|_{\partial D}=u_0|_{\partial D}\;\mathcal{H}^{d-1}\text{-a.e. on }\Gamma\},
%\end{equation*}
%where the symbol $|_{\partial}$ denotes the trace on $\partial D$. We further define the constraint functional $E_{\e,\Gamma,u_0}(\w)$ by
%\begin{equation*}
%	E_{\e,\Gamma,u_0}(\w)(u,D)=\begin{cases}
%		\displaystyle \int_{D\cap S_u}g(\w,\tfrac{x}{\e},u^+,u^-,\nu_u)\,\mathrm{d}\mathcal{H}^{d-1} &\mbox{if $u\in BV_{\Gamma,u_0}(D;\M)$,}
%		\\
%		+\infty &\mbox{otherwise.}
%	\end{cases}
%\end{equation*}
We have the following $\Gamma$-limit of the constrained functionals:
\begin{theorem}\label{thm:boundary}
	Under the same assumptions as in Theorem \ref{thm:main} and the above assumptions on $\Gamma$ and $u_0$, for $\w\in\Omega'$ the functionals $E_{\e,\Gamma,u_0}(\w)(\cdot,D)$ $\Gamma$-converge with respect to the strong convergence in  $L^1(D;\R^m)$ to the functional $E_{\rm hom,\Gamma,u_0}(\omega)(\cdot,D):L^1(D;\R^m)\to [0,+\infty]$ that is finite only on $BV(D;\M)$, where it is given by
	\begin{equation}
		E_{{\rm hom},\Gamma,u_0}(\omega)(u,D)\defas
		\int_{D\cap S_u}g_{\rm hom}(\w,u^+,u^-,\nu_u)\,\mathrm{d}\mathcal{H}^{d-1}+\int_{\Gamma}g_{\rm hom}(\w,u|_{\partial D},u_0|_{\partial D},\nu_{\Gamma})\dHd.
	\end{equation}
	where $g_{\rm hom}$ is given by Theorem \ref{thm:main} and $\nu_{\Gamma}$ is the inner normal vector to $\Gamma$.
\end{theorem} 
\begin{remark}
	The fact that we take the inner normal vector comes from the convention to take $u^+$ has the trace on the side to which the normal vector points to (see Section \ref{s.BV}). The additional part of the energy thus measures the energy of the discrepancy between $u$ and the boundary datum $u_0$ on $\Gamma$. Note that due to the assumption that $\mathcal{H}^{d-1}(S_{u_0}\cap\partial D)=0$ it does not matter if we take the inner or outer trace of $u_0$ on $\partial D$. Moreover, given any function $\bar{u}$ in the trace-space of $BV(D;\M)$, we can use \cite[Lemma 2.7]{BCG} to construct a function $u_0\in BV_{\rm loc}(\R^d;\M)$ such that $u_0|_{\partial D}=\bar{u}$ and $\mathcal{H}^{d-1}(S_{u_0}\cap\partial D)=0$.
\end{remark}
Applying the above theorem to $D=Q^{\nu}$, $\Gamma$ the union of the two faces of $Q^\nu$ orthogonal to $\nu$, and $u_0=u^{a,b,\nu}$, we obtain as a corollary the following convergence result for minimization problems, which in particular shows that isotropic homogenized integrands can be approximated in terms of minimization problems which only require the prescription of boundary conditions on those two faces. More precisely, for every $t>0$ and $(a,b,\nu)\in\M\times\M\times\Sd$ we consider the quantity
\begin{equation}\label{eq:alternativ_process}
	\widetilde{\m}(\omega)(u^{a,b,\nu},Q_t^\nu)\defas\inf\{E_1(\w)(u,Q^{\nu}_t):\,u=u^{a,b,\nu}\text{ on }\{| x\cdot\nu|=\tfrac{t}{2}\}\text{ in the sense of traces}\}.
\end{equation}
Then the following holds true.
\begin{corollary}\label{cor:limit-alternativ-process}
Let $\Omega'$ be as in Theorem~\ref{thm:main} and let $\widetilde{\m}$ be as in~\eqref{eq:alternativ_process}. For every $\omega\in\Omega'$  and every $(a,b,\nu)\in\M\times\M\times\Sd$ the scaled quantities $\frac{1}{t^{d-1}}\widetilde{\m}(\omega)(u^{a,b,\nu},Q_t^\nu)$ converge to
\begin{equation}\label{eq:min-alternative}
	\min_{u\in BV(Q^\nu;\M)}\left\{\int_{Q^{\nu}\cap S_u}g_{\rm hom}(\w,u^+,u^-,\nu_u)\,\mathrm{d}\mathcal{H}^{d-1}+\int_{\{|x \cdot\nu|=\tfrac{1}{2}\}}g_{\rm hom}(\w,u|_{\partial Q^\nu},u^{a,b,\nu},\pm\nu)\dHd\right\}
\end{equation}
as $t\to +\infty$. Moreover, if $g_{\rm hom}$ is isotropic, i.e., independent of $\nu$, then there holds
	\begin{equation}\label{eq:approx-alternative}
	g_{\rm hom}(\omega,a,b)=\lim_{t\to+\infty}\frac{1}{t^{d-1}}\widetilde{\m}(\omega)(u^{a,b,\nu},Q_t^\nu)
	\end{equation}
for every $\omega\in\Omega'$ and every $(a,b,\nu)\in\M\times\M\times\Sd$.
\end{corollary}
\begin{remark}\label{rem:alternative}
Taking $u^{a,b,\nu}$ as a test function in~\eqref{eq:min-alternative} with $u^{a,b,\nu}$, we see that the minimum value in~\eqref{eq:min-alternative} and hence the limit of the scaled quantities $\frac{1}{t^{d-1}}\widetilde{\m}(\omega)(u^{a,b,\nu},Q_t^\nu)$ is bounded from above by $g_{\rm hom}(\omega,a,b,\nu)$. Moreover, in the anisotropic case $g_{\rm hom}(\omega,a,b,\nu)$ is in general strictly larger (see Example~\ref{ex:failure_top_bottom}), while in the isotropic case the above corollary implies that the quantities coincide. Note that isotropy of $g_{\rm hom}$ requires in general strong structural properties of $g$ (e.g., $g$ does not depend on $\nu$ and the distribution of $g$ is statistically isotropic).
\end{remark}
\begin{proof}[Proof of Corollary~\ref{cor:limit-alternativ-process}]
By choosing $\e=\frac{1}{t}$ and applying a change of variables (see, e.g., Remark~\ref{rem:rescaling}) the convergence of $\frac{1}{t^{d-1}}\widetilde{\m}(\omega)(u^{a,b,\nu},Q_t^\nu)$ to~\eqref{eq:min-alternative} follows from Theorem~\ref{thm:boundary} together with the fundamental property of $\Gamma$-convergence and the $L^1$-equi-coercivity of the functionals $E_{\e}(\w)$ whenever $g(\w)\in\A_c$ . 

\smallskip
Let us now assume that $g_{\rm hom}$ is isotropic and let $\omega\in\Omega'$ and $(a,b)\in\M\times\M$ be fixed. It suffices to show that 
\begin{equation}\label{eq:isotropic}
		g_{\rm hom}(\w,a,b)=\min_{u\in BV(Q^\nu;\M)}\left\{\int_{Q^{\nu}\cap S_u}g_{\rm hom}(\w,u^+,u^-)\,\mathrm{d}\mathcal{H}^{d-1}+\int_{\{|x\cdot\nu|=\tfrac{1}{2}\}}g_{\rm hom}(\w,u|_{\partial Q^\nu},u^{a,b,\nu})\dHd\right\}
	\end{equation}
for every $\nu\in\Sd$. In view of Remark~\ref{rem:alternative} it suffices to bound $g_{\rm hom}(\omega,a,b)$ from above by the right-hand side of~\eqref{eq:isotropic}. To this end, it is enough to show that
	\begin{equation}\label{est:isotropic}
		g_{\rm hom}(\w,a,b)\leq \int_{Q^{\nu}\cap S_u}g_{\rm hom}(\w,u^+,u^-)\dHd
	\end{equation}
	for all functions $u\in BV(Q^\nu;\M)$ that attain the boundary values in a neighborhood of the top and the bottom facet as those functions are dense in energy (see Step 3 of the proof of Theorem~\ref{thm:boundary}).
	
	 Since the mapping $x\mapsto g_{\rm hom}(\omega,u^+(x),u^-(x))\mathds{1}_{S_u}$ is Borel measurable, we can apply the coarea formula in~\cite[Theorem 4.1 (a)]{BraidesFD} to deduce that
%	By the general coarea-formula (cf. \cite[(2.72)]{AFP}) applied to the countably $d-1$-rectifiable set $S_u$, the Lipschitz-map 'orthogonal projection onto $H^{\nu}$, and the Borel-function $g_{\rm hom}(\w,u^+(x),u^-(x))$ we can write
	\begin{align*}
		\int_{Q^{\nu}\cap S_u}g_{\rm hom}(\w,u^+,u^-)\dHd&\geq \int_{Q^{\nu}\cap S_u}|\nu_u\cdot\nu|g_{\rm hom}(\w,u^+(x),u^-(x))\dHd(x)
		\\
		&=\int_{Q^{\nu}\cap H^{\nu}}\sum_{t\in S_{u_y^{\nu}}}g_{\rm hom}(\w,(u_y^{\nu})^+,(u_y^{\nu})^-)\dHd(y),
	\end{align*}
	where the function $u_y^{\nu}\in BV((-1/2,1/2);\M)$ is the one-dimensional restriction of $u$ on the slice $(Q^{\nu})_y^{\nu}=\{t\in\R:\,y+t\nu\in Q^{\nu}\}=(-1/2,1/2)$ (cf. Section \ref{s.BV}). Being a $\Gamma$-limit, the functional $E_{\rm hom}(\w)(\cdot,Q^\nu)$ is $L^1$-lower semicontinuous, hence~\cite[Theorem 2.1]{AmBrII} implies that $g_{\rm hom}(\omega,\cdot,\cdot)$ is $BV$-elliptic. Together with~\cite[Example 2.4]{AmBrII} this in turn yields the triangle inequality
%	so that by \cite[(2.10)]{AmBrII} the we have triangle inequality
	\begin{equation}\label{est:triangle}
		g_{\rm hom}(\w,a,b)\leq g_{\rm hom}(\w,a,c)+g_{\rm hom}(\w,c,b)\quad\text{ for all }a,b,c\in \M.
	\end{equation}  
	%Invoking \cite[Example 2.8]{AmBrII} the function $g_{\rm hom}(\w,\cdot)$ is $BV$-elliptic and since $u_z^{\nu_0}=u^{a,b,\nu_0}$ in a neighborhood of $E_z^{\nu_0}$ this implies that
	Since $u=u^{a,b,\nu}$ in a neighborhood of $\big\{|x\cdot\nu|=\frac{1}{2}\big\}$, we know that $u_y^\nu(t)=a$ for $t> 1/2-\delta$ and $u_y^\nu(t)=b$ for $t<-1/2+\delta$ for some $\delta>0$ and $\Hd$-a.e. $y\in H^\nu$. Thus, applying~\eqref{est:triangle} for $\Hd$-a.e. $y\in H^\nu$ we obtain the telescopic estimate
	\begin{equation*}
		g_{\rm hom}(\w,a,b)\leq\sum_{t\in S_{u_y^{\nu}}}g_{\rm hom}(\w,(u_y^{\nu})^+,(u_y^{\nu})^-).
	\end{equation*}
	Integrating this inequality over $H^{\nu}\cap Q^{\nu}$ yields~\eqref{est:isotropic}. 
\end{proof}
% and $u_0=u^{a,b,\nu}$, we deduce that in general we cannot replace the Dirichlet boundary condition used to define $g_{\rm hom}$ by the boundary condition only on the bottom and the top facet of the cube $Q^{\nu}_t$. Indeed, by a change of variables Theorem \ref{thm:boundary} implies that
%\begin{equation}\label{eq:alternativ_process}
%	\frac{1}{t^{d-1}}\inf\{E_1(\w)(u,Q^{\nu}_t):\,u=u^{a,b,\nu}\text{ on }\{| x\cdot\nu|=\tfrac{t}{2}\}\text{ in the sense of traces}\}
%\end{equation}
%converges to
%\begin{equation*}
%	\min_{u\in BV(D;\M)}\left\{\int_{Q^{\nu}\cap S_u}g_{\rm hom}(\w,u^+,u^-,\nu_u)\,\mathrm{d}\mathcal{H}^{d-1}+\int_{\{|x \cdot\nu|=\tfrac{1}{2}\}}g_{\rm hom}(\w,u|_{\partial D},u^{a,b,\nu},\pm\nu)\dHd\right\}.
%\end{equation*}
%Taking $u=u^{a,b,\nu}$ as candidate, this infimum value cannot be greater than $g_{\rm hom}(\w,a,b,\nu)$. Moreover, in general it is strictly smaller as the following example shows.
\begin{example}\label{ex:failure_top_bottom}
	Let $\lambda:\Omega\times\R^d\to [1,2]$ be a stationary map and define the integrand $g(\w,x,\nu)=\lambda(\w,x)\sqrt{8|\nu\cdot e_1|^2+|\nu|^2}$. Note that the term $\psi(\nu)=\sqrt{8|\nu\cdot e_1|^2+|\nu|^2}$ defines a norm on $\R^d$ and thus is convex. Fix now the unit vector $\nu=\frac{1}{\sqrt{2}}e_1+\frac{1}{\sqrt{2}}e_2$. Using the lower bound $\lambda\geq 1$ and the BV-ellipticity of the map $\nu\mapsto \psi(\nu)$ (see \cite[Example 2.8]{AmBrII}), we obtain that
	\begin{align*}
	\frac{1}{t^{d-1}}\inf\big\{E_1(\omega)(u,Q_t^\nu) &\colon u\in BV(Q_t^\nu;\M),\ u=u_0^{a,b,\nu}\ \text{near}\ \partial Q_t^\nu\big\}\\
	&\geq\frac{1}{t^{d-1}}\inf\bigg\{\int_{Q_t^\nu\cap S_u}\psi(\nu_u)\dHd\colon u\in BV(Q_t^\nu;\M),\ u=u_0^{a,b,\nu}\ \text{near}\ \partial Q_t^\nu\bigg\}\\
		&\geq\frac{1}{t^{d-1}}\int_{Q^{\nu}_t\cap  S_{u_0^{a,b,\nu}}}\psi(\nu_{u_0^{a,b,\nu}})\dHd=\sqrt{8|\nu_1|^2+|\nu|^2}=\sqrt{5},
	\end{align*}
	so that $g_{\rm hom}(\w,a,b,\nu)\geq \sqrt{5}>2$. On the other hand, the function $u_0^{a,b,e_2}$ satisfies $u_0^{a,b,e_2}=u_0^{a,b,\nu}$ on $\{|x\cdot\nu|=\tfrac{1}{2}\}$ and using that $\lambda\leq 2$ its energy can be bounded by $2$. Note that this effect is only due to anisotropy. Indeed, the same example works for the case that $\lambda\equiv 1$.
\end{example}
\section{Proof of the main results}
\subsection{Existence of the multi-cell formula}
As a first step towards the proof of Theorem~\ref{thm:main} we show that the limit defining $g_{\rm hom}$ in~\eqref{def:ghom} exists almost surely. To this end, let us introduce for every $\e>0$, every $A\in\mathscr{A}$, and every $v\in BV_{\rm loc}(\R^d;\M)$ the quantity
	\begin{equation}\label{def:min-problem}
	\m_\e(\omega)(v,A)\defas\inf\big\{E_\e(\omega)(u,A)\colon u\in\Adm(v,A)\big\}. 
	\end{equation}
with
	\begin{equation}
	\Adm(v,A)\defas\big\{u\in BV(A;\M)\colon u=v\ \text{near}\ \partial A\}
	\end{equation}
The following holds true.
\begin{proposition}\label{prop:existence-limit}
Let $g:\Omega\times\R^d\times\M\times\M\times\Sd\to[0,+\infty)$ be an admissible random surface tension which is $\Z^d$-stationary. There exists $\Omega'\subset\Omega$ with $\P(\Omega')=1$ such that for every $\omega\in\Omega'$ and every $(x,a,b,\nu)\in\R^d\times\M\times\M\times\Sd$ the limit
	\begin{equation}\label{eq:existence-limit}
	\begin{split}
	g_{\rm hom}(\omega,a,b,\nu)=
	\lim_{t\to+\infty}\frac{\m_1(\omega)(u_{t x}^{a,b,\nu},Q_t^\nu(tx))}{t^{d-1}}
	=\lim_{t\to+\infty}\frac{\m_1(\omega)(u_{0}^{a,b,\nu},Q_t^\nu)}{t^{d-1}}
	\end{split}
	\end{equation}
exists and is independent of $x$. 
%More precisely, there exists an admissible random surface tension $g_{\rm hom}:\Omega\times\M\times\M\times\Sd$ such that
%	\begin{equation*}
%	g_{\rm hom}(\omega,a,b,\nu)=
%	\lim_{t\to+\infty}\frac{\m_1(\omega)(u_{t x}^{a,b,\nu},Q_t^\nu(tx))}{t^{d-1}}
%	=\lim_{t\to+\infty}\frac{\m_1(\omega)(u_{0}^{a,b,\nu},Q_t^\nu)}{t^{d-1}}
%	\end{equation*}
%for every $\omega\in\Omega'$ and every $(x,a,b,\nu)\in\R^d\times\M\times\M\times\Sd$.
Moreover, the function $(a,b,\nu)\mapsto g_{\rm hom}(\w,a,b,\nu)$ is continuous. Eventually, if $g$ is ergodic then $g_{\rm hom}$ is independent of $\omega$ and
	\begin{equation}\label{eq:ghom-ergodic}
	g_{\rm hom}(a,b,\nu)=\lim_{t\to+\infty}\frac{\mathbb{E}\big[\m_1(\omega)(\uzn,Q_t^\nu)\big]}{t^{d-1}}.
	\end{equation}
\end{proposition}
\begin{remark}\label{rem:rescaling}
Let $\Omega'$ be as in Proposition~\ref{prop:existence-limit} and let $\omega\in\Omega'$ and $(a,b,\nu)\in\M\times\M\times\Sd$ be arbitrary. Then~\eqref{eq:existence-limit} together with a change of variables implies that
	\begin{equation}\label{eq:limit-rescaled}
	g_{\rm hom}(\omega,a,b,\nu)=\frac{1}{\rho^{d-1}}\lim_{\e\to 0}\m_\e(\omega)(u_{x}^{a,b,\nu},Q_\rho^\nu(x))
	\end{equation}
for every $x\in\R^d$ and $\rho>0$. Indeed, for every $u\in BV (Q_\rho^\nu(x);\M)$ and every $\e>0$ let $u_\e\in BV\big(Q_{\frac{\rho}{\e}}^\nu\big(\frac{x}{\e}\big);\M\big)$ be given by $u_\e(y)\defas u(\e y)$. Then $u\in\Adm(u_{x}^{a,b,\nu};Q_\rho^\nu(x))$ if and only if $u_\e\in\Adm\big(u_{\frac{x}{\e}}^{a,b,\nu},Q_{\frac{\rho}{\e}}^\nu\big(\frac{x}{\e}\big)\big)$ and $S_{u_\e}=\frac{1}{\e}S_u$. Moreover, the change of variables $z=\frac{y}{\e}$ yields $E_\e(\omega)(u,Q_\rho^\nu(x))=E_1(\omega)(u_\e, Q_{\frac{\rho}{\e}}^\nu(\frac{x}{\e}))$. Setting $t_\e\defas\frac{\rho}{\e}$, $x^\rho\defas \rho x$ and passing to the infimum in $u$ we thus obtain 
	\begin{equation*}
	\m_\e(\omega)\big(u_{x}^{a,b,\nu},Q_\rho^\nu(x)\big)=\Big(\frac{\rho}{t_\e}\Big)^{d-1}\m_1(\omega)\big(u_{t_\e x^\rho}^{a,b,\nu},Q_{t_\e}^\nu(t_\e(x^\rho)\big).
	\end{equation*}
Hence~\eqref{eq:limit-rescaled} follows by dividing the above inequality by $\rho^{d-1}$ and applying Proposition~\ref{prop:existence-limit} with $t=t_\e$.
\end{remark}
We will prove Proposition~\ref{prop:existence-limit} first in the case $x=0$ by associating a suitable subadditive stochastic process to the minimization problem $\m_1(\omega)(\uzn,Q_t^\nu)$ and applying Theorem~\ref{thm:subadditive}.
%We now associate a subadditive process to $\m_1(\omega)(u_0^{a,b,\nu},Q_t^\nu)$ as follows. 
To this end, for fixed $\nu\in\Sph^{d-1}\cap\Q^{d}$ we let $O_\nu$ be the orthogonal matrix induced by~\eqref{eq:matrix}. 
Since $\nu\in\mathbb{S}^{d-1}\cap\mathbb{Q}^d$ is a rational direction it follows that $O_\nu\in\mathbb{Q}^{d\times d}$, so that there exists an integer $m_\nu$ such that $m_\nu O_\nu(z,0)\in\Z^d$ for every $z\in\Z^{d-1}$.
For every $I=[p_1,q_1)\times\cdots\times[p_{d-1},q_{d-1})\in\mathcal{I}$ we denote by $s_{\rm max}(I)\defas\max_{i}|q_i-p_i|$ its maximal side length and define the open set $I^\nu\subset\R^d$ as
\begin{equation}\label{def:ndimintervals}
I^\nu\defas m_\nu O_\nu\Big({\rm int}\, I\times s_{\rm max}(I)(-1/2,1/2)\Big),
\end{equation}
where ${\rm int}$ denotes the $(d-1)$-dimensional interior. 
Then we define a function $\mu^{a,b,\nu}:\mathcal{I}\times\Omega\to\R$ by setting
\begin{equation}\label{def:subadprocess}
\mu^{a,b,\nu}(I,\omega)\defas\frac{1}{m_\nu^{d -1}}\inf\{E_1(\omega) (u,I^\nu)\colon u\in\Adm(u_0^{a,b,\nu},I^\nu)\}.
\end{equation}
\begin{lemma}\label{lem:subadditive-process}
Let $g:\Omega\times\R^d\times\M\times\M\times\Sph^{d-1}\to[0,+\infty)$ be an admissible random surface tension which is $\Z^d$-stationary. For every $(a,b,\nu)\in\M\times\M\times(\Sd\cap\Q^d)$ let $\mu^{a,b,\nu}(I,\omega)$ be as in~\eqref{def:subadprocess}. Then there exists a measure-preserving group action $\{\tau_z^\nu\}_{z\in\Z^{d-1}}$ such that $\mu^{a,b,\nu}$ is a subadditive process with respect to $\{\tau_{z}^\nu\}_{z\in\Z^{d-1}}$ satisfying
\begin{equation}\label{eq:l_inftybound}
0\leq\mu^{a,b,\nu}(I,\omega)\leq c  \mathcal{L}^{d-1}(I)
\end{equation}
for a.e. $\omega\in\Omega$ and every $I\in\I$.
\end{lemma}
\begin{proof}
Let $(a,b,\nu)\in\M\times\M\times(\Sd\cap\Q^d)$ be fixed; the fact that for every $I\in\I$ the map $\omega\mapsto\mu^{a,b,\nu}(I,\omega)$ is $\F$-measurable, follows by applying Proposition~\ref{prop:measurable} with $A=I^\nu$ and $v=u_0^{a,b,\nu}$. 
Moreover, since $g(\omega)\in\A_c$ for a.e. $\omega\in\Omega$, we obtain~\eqref{eq:l_inftybound} by testing the minimization problem with the function $u_0^{a,b,\nu}$ itself.
It remains to show that $\mu^{a,b,\nu}:\I\times\Omega\to [0,+\infty)$ is stationary and subadditive. 
%This can be done following the lines of~\cite[Proposition 5.3]{CDMSZ19} (see also~\cite[Lemma 5.2]{BaRu22}) \RRR[[Leave the citations out. If you want to acknowledge the idea, then anyway \cite{ACR} was the first to define this process (Piatnitsky once told me that he doesn't know which process to take for $d> 2$)]]\EEE. To be self-contained, we prefer to include the short proof here.

\smallskip 
To prove the stationarity of the process, we define for every $z\in\Z^{d-1}$ the vector $z^\nu:=m_\nu O_\nu(z,0)\in\Z^d$. Setting  $\tau_z^\nu:=\tau_{z^\nu}$ we obtain a discrete measure-preserving group action $\{\tau_z^\nu\}_{z\in\Z^{d-1}}$. Note that for every $I\in\mathcal{I}$ and every $z\in\R^{d-1}$ we have $(I+z)^\nu=I^\nu+z^\nu$. For every $u\in BV((I+z)^\nu;\M)$ let us set $u_z\defas u(\cdot\,+z^\nu)$. It is immediate to check that $u\in \Adm(u_0^{a,b,\nu},(I+z)^\nu)$ if and only if $u_z\in\Adm(u_0^{a,b,\nu},I^\nu)$. 
Moreover, the stationarity of $g$ with respect to $\{\tau_z\}_{z\in\Z^d}$ together with a change of variables yields $E_1(\omega) (u,(I+z)^\nu)=E_1(\tau_z^\nu\omega)(u_z,I^\nu)$.
Thus, we conclude by minimization that $\mu^{a,b,\nu}(I+z,\omega)=\mu^{a,b,\nu}(I,\tau_z^\nu\omega)$, which implies the stationarity of the process with respect to the lower-dimensional group action $\{\tau_z^\nu\}_{z\in\mathbb{Z}^{d-1}}$. 

\smallskip
To show that $\mu^{a,b,\nu}$ is subadditive, let $I\in\mathcal{I}$ be arbitrary and let $(I_i)_{i=1}^k\subset\mathcal{I}$ be pairwise disjoint and such that $I=\bigcup_{i=1}^k I^i$. Fix $\eta>0$ and for any $i\in\{1,\ldots, k\}$ let $u_i\in \Adm(\uzn,I_i^\nu)$ be such that 
\begin{align}\label{subad:almost optimal}
\Egw (u_i, I_i^\nu)\leq \mu^{a,b,\nu}(I_i,\omega)+k^{-1}\eta.
\end{align}
Since also the $d$-dimensional cuboids $I_i^\nu$ are pairwise disjoint, we can define a function $u\in BV(I^\nu;\M)$ by setting 
\begin{equation*}
u(x):=
\begin{cases}
u_i(x) &\text{if}\ x\in I_i^\nu\; \text{ for some}\ 1\leq i\leq k,\\
\uzn(x) &\text{otherwise}.
\end{cases}
\end{equation*}
Since $s_{\rm max}(I_i)\leq s_{\rm max}(I)$ for all $1\leq i\leq k$, all cuboids $I_i^\nu$ are contained in $I^\nu$, so that the function $u$ belongs to $\Adm(\uzn,I^\nu)$. Moreover, since $S_{\uzn}=H^\nu$, thanks to the boundary conditions satisfied by each $u_i$ and the equality $I=\bigcup_{i=1}^kI_i$ we have $S_u\cap I^\nu=S_u\cap\big(\bigcup_{i=1}^k\overline{I_i^\nu}\big)=\bigcup_{i=1}^kS_{u_i}$. Thus, using the additivity of $\Egw$ as a set function, from~\eqref{subad:almost optimal} we infer
\begin{align*}
\mu^{a,b,\nu}(I,\omega)\leq \Egw (u,I^\nu)=\sum_{i=1}^k \Egw (u_i,I_i^\nu)\leq\mu^{a,b,\nu}(I_i,\omega)+\eta
\end{align*}
and we conclude by the arbitrariness of $\eta>0$.
\end{proof}
We are now in a position to prove Proposition~\ref{prop:existence-limit}.
\begin{proof}[Proof of Proposition~\ref{prop:existence-limit}]
The proof is divided in several steps. We start by establishing the existence of the limit for $x=0$ and any fixed rational direction $\nu\in\Sd\cap\Q^d$.

\smallskip
\textbf{Step 1.} We show that there exists $\widehat{\Omega}\subset\Omega$ with $\P(\widehat\Omega)=1$ and such that for every $\omega\in\widehat{\Omega}$ we have
	\begin{equation}\label{eq:existence-rational}
	g_{\rm hom}(\omega,a,b,\nu)=\lim_{t\to+\infty}\frac{1}{t^{d-1}}\m_1(\omega)(\uzn,Q_t^\nu)\;\text{ for every}\ (a,b,\nu)\in\M\times\M\times\Sd\cap\Q^d.
	\end{equation}
Thanks to Lemma~\ref{lem:subadditive-process}, for fixed $\nu\in\Sd\cap\Q^d$ and $(a,b)\in\M\times\M$ we can apply Theorem~\ref{thm:subadditive}
%the subadditive ergodic theorem~\cite[Theorem 2.4]{AkKr} (see~\cite[Proposition 1]{DMMoII} for the version used here) 
to deduce the existence of a set $\Omega^{a,b,\nu}\subset\Omega$ with $\P(\Omega^{a,b,\nu})=1$ and of an $\F$-measurable function $g_{\rm hom}(\cdot\,,a,b,\nu):\Omega\to[0,+\infty)$ such that~\eqref{eq:existence-rational} holds true for every $\omega\in\Omega^{a,b,\nu}$. The set $\widehat{\Omega}$ is then obtained by taking the countable intersection over all $\Omega^{a,b,\nu}$ with $(a,b)\in\M\times\M$ and $\nu\in\Sd\cap\Q^d$. 
%It is moreover not restrictive to assume that $g(\omega)\in\A_c$ for every $\omega\in\widehat{\Omega}$.

\smallskip
\textbf{Step 2.} We show that~\eqref{eq:existence-rational} still holds for $\nu\in\Sd\setminus\Q^d$ and $\omega\in\widehat{\Omega}$. To this end, for every $\nu\in\Sd$ and fixed $(a,b)\in\M\times\M$ we set
	\begin{equation}\label{def:lim-sup-inf}
	\overline{g}(\omega,\nu)\defas\limsup_{t\to+\infty}\frac{\m_1(\omega)(\uzn,Q_t^\nu)}{t^{d-1}}\quad\text{and}\quad\underline{g}(\omega,\nu)\defas\liminf_{t\to+\infty}\frac{\m_1(\omega)(\uzn,Q_t^\nu)}{t^{d-1}}.
	\end{equation}
Note that $\overline{g}(\omega,\nu)=\underline{g}(\omega,\nu)$ for every $\omega\in\widehat{\Omega}$ and every $\nu\in\Sd\cap\Q^d$. In view of Step 1 it therefore suffices to show that for every $\omega\in\widehat{\Omega}$ the restriction of $\overline{g}(\omega,\cdot)$ and $\underline{g}(\omega,\cdot)$ to $\Sd_+\defas\Sd\sm\{-e_d\}$ is continuous. Then we can argue by density of $\Sd_+\cap\Q^d$ in $\Sd_+$. Moreover, as a consequence we obtain the required continuity of $g_{\rm hom}$ at least on $\Sd_+$. 
%We only prove the result for $\overline{g}(\omega,\cdot)$ restricted to $\Sd\sm\{e_d\}$, the other cases being analogous. 
Let $\nu\in\Sd_+$ and let $(\nu_j)\subset\Sd_+$ such that $\nu_j\to\nu$ as $j\to+\infty$ and let $\alpha\in (0,\frac{1}{2})$ be arbitrary. By the continuity of the map $\nu\mapsto O_\nu$ on ${\Sph}^{d-1}_+$ we have that
	\begin{equation}\label{inclusion:cubes_0}
	Q_{(1-\alpha)t}^{\nu_j}\wcont Q_t^\nu\wcont Q_{(1+\alpha)t}^{\nu_j}
	\end{equation}
for every $t>0$ and for $j$ sufficiently large (depending on $\alpha$). Using the second inclusion in~\eqref{inclusion:cubes_0} we can extend any competitor $u\in\Adm(\uzn, Q_t^\nu)$ to a competitor $u\in\Adm(u_0^{a,b,\nu_j},Q_{(1+\alpha)t}^{\nu_j})$ by setting $u\defas u_0^{a,b,\nu_j}$ on $Q_{(1+\alpha)t}^{\nu_j}\setminus Q_t^\nu$. Since $g(\omega)\in\A_c$, we have
	\begin{equation}\label{est:extension-competitor}
	E_1(\omega)(u,Q_{(1+\alpha)t}^{\nu_j})\leq E_1(\omega)(u,Q_t^\nu)+c\Big(\Hd\big(H^\nu\cap (Q_{(1+\alpha)t}^{\nu_j}\setminus Q_t^\nu)\big)+\Hd\big(\partial Q_t^\nu\cap\{u_0^{a,b,\nu}\neq u_{0}^{a,b,\nu_j}\}\big)\Big).
	\end{equation}
Note that $\Hd\big(H^\nu\cap (Q_{(1+\alpha)t}^{\nu_j}\setminus Q_t^\nu)\big)\leq c_\alpha t^{d-1}$ and $\Hd\big(\partial Q_t^\nu\cap\{u_0^{a,b,\nu}\neq u_{0}^{a,b,\nu_j}\}\big)\leq c_j  t^{d-1}$ for every $t>0$ with $c_\alpha\to 0$ as $\alpha\to 0$ and $c_j\to 0$ as $j\to +\infty$. Passing to the infimum over $u$, dividing by  $((1+\alpha)t)^{d-1}$  and passing to the limsup as $t\to+\infty$ we thus deduce from~\eqref{est:extension-competitor} that
	\begin{equation*}
	\overline{g}(\omega,\nu_j)\leq \overline{g}(\omega,\nu)+c_\alpha+c_j.
	\end{equation*}
Letting first $j\to+\infty$ and then $\alpha\to 0$ this implies that $\limsup_{j}\overline{g}(\omega,\nu_j)\leq\overline{g}(\omega,\nu)$ A similar argument now using the first inclusion in~\eqref{inclusion:cubes_0} yields $\overline{g}(\omega,\nu)\leq\liminf_j\overline{g}(\omega,\nu_j)$, which yields the continuity of $\overline{g}(\omega,\cdot)$ on ${\Sph}^{d-1}_+$. The argument for $\underline{g}(\w,\cdot)$ is analogous.

\smallskip
\textbf{Step 3.} As an intermediate step we prove the invariance of $\widehat{\Omega}$ and $g_{\rm  hom}$ under $\{\tau_z\}_{z\in\Z^d}$, that is, we show that $\tau_z(\widehat{\Omega})=\widehat{\Omega}$ and
	\begin{equation}\label{eq:invariance-ghom}
	g_{\rm hom}(\tau_z\omega,a,b,\nu)=g_{\rm hom}(\omega,a,b,\nu)
	\end{equation}
for every $\omega\in\widehat{\Omega}$ and every $(a,b,\nu)\in\M\times\M\times\Sd$.

For $(a,b,\nu)\in\M\times\M\times\Sd$ fixed let $\overline{g},\underline{g}$ be as in~\eqref{def:lim-sup-inf} and let us show that
	\begin{equation}\label{eq:ghom-shift}
	\overline{g}(\tau_z\omega,\nu)\leq g_{\rm hom}(\omega,a,b,\nu)\leq \underline{g}(\tau_z\omega,\nu) \; \text{ for every $ \omega\in\widehat{\Omega}$ and $z\in\Z^d$}.
	\end{equation}	 
For any $z\in\Z^d$ and any $u\in BV(Q_t^\nu;\M)$ set $u_z(x)\defas u(x+z)$ for every $x\in Q_t^\nu(-z)$. Then $u\in\Adm(u_0^{a,b,\nu},Q_t^\nu)$ if and only if $u_z\in\Adm(u_{-z}^{a,b,\nu},Q_t^\nu(-z))$. Moreover, a change of variables together with the stationarity of $g$ gives
	\begin{equation}\label{eq:shift}
	E_1(\omega)(u,Q_t^\nu)=E_1(\tau_z\omega)(u_z,Q_t^\nu(-z)).
	\end{equation}
We now choose $\tilde{t}>t+2|z|$, so that $Q_t^\nu(-z)\wcont Q_{\tilde{t}}^\nu$ and we extend $u_z$ to $Q_{\tilde{t}}^\nu$ by setting $u_z\defas \uzn$ in $Q_{\tilde{t}}^\nu\setminus Q_t^\nu(-z)$. Then $u_z\in\Adm(\uzn,Q_{\tilde{t}}^\nu)$ and the boundary conditions satisfied by $u_z$ on $\partial Q_t^\nu(-z)$ together with the bounds on $g$ ensure that 
	\begin{equation}\label{est:extension-shift}
	\begin{split}
	E_1(\tau_z\omega)(u_z,Q_{\tilde{t}}^\nu) &\leq E_1(\tau_z\omega)(u_z,Q_t^\nu(-z))\\
	&\hspace*{1em}+c\Big(\Hd\big( H^\nu \cap(Q_{\tilde{t}}^\nu\setminus Q_t^\nu(-z))\big)+\Hd\big(\partial Q_t^\nu(-z)\cap\{u_{-z}^{a,b,\nu}\neq u_0^{a,b,\nu}\}\big)\Big)\\
	&\leq E_1(\tau_z\omega)(u_z,Q_t^\nu(-z))+C(|z|+|\tilde{t}-t|+1)\tilde{t}^{d-2}.
	\end{split}
	\end{equation}
Combining~\eqref{eq:shift} and~\eqref{est:extension-shift}, passing to the infimum in $u$, dividing by $\tilde{t}^{d-1}$, and passing to the limsup in $\tilde{t}$ and the limit in $t$ we deduce that $\overline{g}(\tau_z\omega,\nu)\leq g_{\rm hom}(\omega,a,b,\nu)$. By exchanging the roles of $\omega$ and $\tau_z\omega$ in~\eqref{eq:shift} and replacing $z$ by $-z$ we obtain in a similar way that $g_{\rm hom}(\omega,a,b,\nu)\leq\underline{g}(\tau_z\omega,\nu)$, hence~\eqref{eq:ghom-shift} is proved. From~\eqref{eq:ghom-shift} we deduce that $\tau_z\widehat{\Omega}\subset\widehat{\Omega}$ for every $z\in\Z^d$, while the opposite inclusion follows from the group property. Moreover,~\eqref{eq:ghom-shift} implies~\eqref{eq:invariance-ghom}.

\smallskip
If $\{\tau_z\}_{z\in\Z^d}$ is ergodic, then~\eqref{eq:invariance-ghom} implies that $g_{\rm hom}$ is independent of $\omega$. Then~\eqref{eq:ghom-ergodic} follows thanks to~\eqref{eq:l_inftybound} by applying the dominated convergence theorem.

\smallskip
\textbf{Step 4.} Next, we consider the case $x\neq 0$. Consider for the moment a cube $Q_\rho^{\nu}(x)$ with rational direction $\nu\in\mathbb{S}^{d-1}\cap\Q^d$, $x\in\mathbb{Z}^d\backslash\{0\}$ and $\rho\in\mathbb{Q}$. We show that there exists $\Omega_{\rho,x}^\nu\subset\widehat{\Omega}$ with $\P(\Omega_{\rho,x}^\nu)=1$ and such that
	\begin{equation}\label{existence:limit-x-prelim}
	\lim_{k\to+\infty}\frac{\m_1(\omega)(u_{kx}^{a,b,\nu},Q_{k\rho}(kx))}{(k\rho)^{d-1}}=g_{\rm hom}(\omega,a,b,\nu)\;\text{ for every}\ \omega\in\Omega_{\rho,x}^\nu.
	\end{equation}
To this end, given $\delta>0$ and $K\in\N$, we define the events
\begin{equation*}
	\mathcal{Q}_{K,\delta}:=\left\{\w\in\Omega:\;\sup_{t\geq\frac{K}{2}}\left|(t\rho)^{1-d}\m_1(\w)(u_{0}^{a,b,\nu},Q^{\nu}_{t\rho})-g_{\rm hom}(\w,a,b,\nu)\right|\leq\delta\right\}.
\end{equation*}
From Steps 1+2 we know that $\mathds{1}_{\mathcal{Q}_{K,\delta}}(\omega)\to 1$ for every $\omega\in\widehat\Omega$ when $K\to +\infty$. Let us denote by $\mathcal{J}_x$ the $\sigma$-algebra of invariant sets for the measure preserving map $\tau_x$. The monotone convergence theorem for the conditional expectation yields
\begin{equation*}
	1=\mathbb{E}[1|\mathcal{J}_x]=\lim_{K\to +\infty}\mathbb{E}[\mathds{1}_{\mathcal{Q}_{K,\delta}}|\mathcal{J}_x],
\end{equation*} 
so that almost surely we find $K_0=K_0(\w,\delta)$ such that $1-\delta\leq\mathbb{E}[\mathds{1}_{\mathcal{Q}_{K_0,\delta}}|\mathcal{J}_x](\w)\leq 1$.

Next, due to Birkhoff's ergodic theorem, almost surely, there exists $k_0=k_0(\w,\delta)$ such that, for any $k\geq \frac{k_0}{2}$,
\begin{equation}\label{est:Birckhoff}
	\left|\frac{1}{k}\sum_{i=1}^k\mathds{1}_{\mathcal{Q}_{K_0,\delta}}(\tau_{i x}\w)-\mathbb{E}[\mathds{1}_{\mathcal{Q}_{K_0,\delta}}|\mathcal{J}_x](\w)\right|\leq\delta.
\end{equation}
More precisely, we find $\Omega_{x,\rho}^{\nu,\delta}$ such that~\eqref{est:Birckhoff} holds for every $\omega\in\Omega_{x,\rho}^{\nu,\delta}$.
%Note that the set we exclude will be a countable union of null sets provided $\delta$ is rational. 

For fixed $k\geq\max\{k_0,K_0\}$ we denote by $N_k$ the maximal integer such that for all $i=k+1,\dots,k+N_k$ we have $\tau_{ix}(\w)\notin \mathcal{Q}_{K_0,\delta}$. In order to bound $N_k$ let $\tilde{k}$ be the number of non-zero terms in the sequence $\{\mathds{1}_{\mathcal{Q}_{K_0,\delta}}(\tau_{ix}(\w))\}_{i=1}^{k}$. The definition of $N_k$ implies that
\begin{equation*}
	\delta\geq\left|\frac{\tilde{k}}{k+N_k}-\mathbb{E}[\mathds{1}_{\mathcal{Q}_{K_0,\delta}}|\mathcal{J}_x](\w)\right|=\left|1-\mathbb{E}[\mathds{1}_{\mathcal{Q}_{K_0,\delta}}|\mathcal{J}_x](\w)-\frac{k+N_k-\tilde k}{k+N_k}\right|\geq \frac{k+N_k-\tilde{k}}{k+N_k}-\delta.
\end{equation*}
Since $k-\tilde{k}\geq 0$ and without loss of generality $\delta<\frac{1}{4}$, we obtain the upper bound $N_k< 4k\delta$. 
Hence for any $k\geq\max\{k_0,K_0\}$ and $\tilde{N}_k=\lceil 4k\delta\rceil$ there exists $i_k=i_k(\delta)\in \{k+1,\ldots,k+\tilde{N}_k\}$ such that $\tau_{i_kx}(\w)\in \mathcal{Q}_{K_0,\delta}$. Together with the stationarity of $g$ and~\eqref{eq:invariance-ghom} in Step 3 this implies that
\begin{equation}\label{eq:epsestimate}
	\left|(t\rho)^{1-d}\m_1(\w)(u_{i_kx}^{a,b,\nu},Q^{\nu}_{t\rho}(i_kx))-g_{\rm hom}(\w,a,b,\nu)\right|\leq\delta\;\text{ for all $t\geq \frac{K_0}{2}$.}
\end{equation}
Recall that $|i_k-k|\leq \lceil 4k\delta\rceil$. In what follows, we estimate the error when we replace $i_k$ by $k$ in the above estimate. Define $\alpha_k\defas k-2\rho^{-1}|x|(i_k-k)$ and $\beta_k\defas k+2\rho^{-1}|x|(i_k-k)$, so that 
	\begin{equation}\label{inclusion:cubes}
	Q_{\alpha_k\rho}^\nu(i_kx)\wcont Q_{k\rho}^\nu(kx)\wcont Q_{\beta_k\rho}^\nu(i_kx).
	\end{equation}
Extending any competitor $u\in\Adm(u_{kx}^{a,b,\nu},Q^{\nu}_{k\rho}(kx))$ onto $Q_{\beta_k\rho}^{\nu}(i_kx)$ by $u_{i_kx}^{a,b,\nu}$, we find that
\begin{align}\label{eq:almost_kx}
	\m_1(\w)(u_{i_kx}^{a,b,\nu},Q_{\beta_k\rho}^{\nu}(i_kx))\leq \m_1(\w)(u_{kx}^{a,b,\nu},Q^{\nu}_{k\rho}(kx))&+C |(i_kx-kx)\cdot\nu|(k\rho)^{d-2}\nonumber
	\\
	&+C\mathcal{H}^{d-1}(H^{\nu}(i_kx)\cap Q_{\beta_k\rho}^{\nu}(i_kx)\setminus  Q^{\nu}_{k\rho}(kx)),
\end{align} 
where the first error term stems from the newly created discontinuity set on $\partial Q^{\nu}_{k\rho}(kx)$. To control the second error term, note that thanks to~\eqref{inclusion:cubes} we have
	\begin{align*}
	\mathcal{H}^{d-1}(H^{\nu}(i_kx)\cap Q_{\beta_k\rho}^{\nu}(i_kx)\setminus  Q^{\nu}_{k\rho}(kx)) &\leq\Hd(H^\nu(i_kx)\cap Q_{\beta_k\rho}^\nu(i_kx)\sm Q_{\alpha_k\rho}^\nu(i_kx))\\
	&\leq C\rho^{d-1}(\beta_k^{d-1}-\alpha_k^{d-1})\leq C(k\rho)^{d-2}|x|(i_k-k),
	\end{align*}
provided $k$ is sufficiently large such that $k\rho\geq |x|(i_k-k)$. Inserting also the bound $i_k-k\leq \lceil 4k\delta\rceil$ and $\beta_k\geq k$, and dividing by $(\beta_k\rho)^{d-1}$ we continue the estimate \eqref{eq:almost_kx} to obtain
\begin{equation}\label{eq:firstbound}
	\frac{\m_1(\w)(u_{i_kx}^{a,b,\nu},Q_{\beta_k\rho}^{\nu}(i_kx))}{(\beta_k\rho)^{d-1}}\leq \frac{\m_1(\w)(u_{kx}^{a,b,\nu},Q^{\nu}_{k\rho}(kx))}{(k\rho)^{d-1}}+C\frac{\lceil 4k\delta\rceil}{\rho k}|x|.
\end{equation}
Using instead the first inclusion in~\eqref{inclusion:cubes} the same argument as above leads to
\begin{equation}\label{eq:otherbound}
	\frac{\m_1(\w)(u_{kx}^{a,b,\nu},Q^{\nu}_{k\rho}(kx))}{(k\rho)^{d-1}}\leq\frac{\m_1(\w)(u_{i_kx}^{a,b,\nu},Q^{\nu}_{\alpha_k\rho}(i_kx))}{(\alpha_k\rho)^{d-1}}+C\frac{\lceil 4k\delta\rceil}{\rho k}|x|. 
\end{equation}
Now if $\delta$ is small enough (depending only on $x$ and $\rho$) we have $\beta_k\geq \alpha_k\geq\frac{k}{2}\geq\frac{K_0}{2}$, so that \eqref{eq:epsestimate} holds for the choices $t=\alpha_k$ and $t=\beta_k$. Combining \eqref{eq:epsestimate} with \eqref{eq:firstbound}, \eqref{eq:otherbound} and the triangle inequality, we infer
\begin{equation*}
	\limsup_{k\to +\infty}\left|\frac{\m_1(\w)(u_{kx}^{a,b,\nu},Q^{\nu}_{k\rho}(kx))}{(k\rho)^{d-1}}-g_{\rm hom}(\w,a,b,\nu)\right|\leq \frac{4C\delta}{\rho}|x|\;\text{ for every}\ \omega\in\Omega_{x,\rho}^{\nu,\delta}.
\end{equation*}
Choosing $\Omega_{x,\rho}^\nu$ to be the intersection over all $\Omega_{x,\rho}^{\nu,\delta}$ with $\delta\in\Q_+$ we thus obtain~\eqref{existence:limit-x-prelim}. 

\smallskip
We finally define $\Omega'$ to be the countable intersection over all $\Omega_{x,\rho}^\nu$ with $x\in\Z^d\sm\{0\}$, $\rho\in\Q_+$ and $\nu\in\Sd\cap\Q^d$, so that $\omega\in\Omega'$ satisfy~\eqref{existence:limit-x-prelim} simultaneously for all such $x,\rho,\nu$. For fixed $x\in\Z^d\sm\{0\}$, $\rho\in\Q_+$, and $\nu\in\Sd\cap\Q^d$ we extend~\eqref{existence:limit-x-prelim} to arbitrary sequences $t_k\to+\infty$ as follows: for any $\delta\in(0,1)\cap\Q$ and $t_k$ sufficiently large such that $1-\delta<(\tfrac{t_k}{t_k+1}-\frac{2|x|}{\rho (t_k+1)})$ and $1+\frac{2|x|}{\rho t_k}<1+\delta$ we have 
$$Q^{\nu}_{(1-\delta)\rho\lceil t_k\rceil}(\lceil t_k\rceil x)\wcont Q_{\rho t_k}^\nu(t_kx)\wcont Q_{(1+\delta)\rho\lceil t_k\rceil}^\nu(\lceil t_k\rceil x).$$ 
Thus, a similar comparison argument as in~\eqref{eq:firstbound}-\eqref{eq:otherbound} gives 
	\begin{align*}
	\m_1(\omega)\big(u_{\lceil t_k\rceil x}^{a,b,\nu},Q_{(1+\delta)\rho\lceil t_k\rceil}^\nu(\lceil t_k\rceil x)\big)-C\delta (\rho t_k)^{d-1}|x| &\leq\m_1\big(u_{t_kx}^{a,b,\nu},Q_{\rho t_k}(t_kx)\big)\\
	&\leq\m_1(\omega)\big(u_{\lfloor t_k\rfloor x}^{a,b,\nu},Q^{\nu}_{(1-\delta)\rho\lceil t_k\rceil}(\lceil t_k\rceil x)\big)+C\delta (\rho t_k)^{d-1}|x|.
	\end{align*}
Applying~\eqref{existence:limit-x-prelim} and letting $\delta\to 0$ thus gives the desired convergence along the sequence $(t_k)$. This further allows us to extend the convergence to rational centers $x$ be writing $x=qz$ with $z\in\Z^d$ and $q\in\Q_+$, replacing $t_k$ by $qt_k$ and $\rho$ by $\frac{\rho}{q}\in\Q$ in~\eqref{existence:limit-x-prelim}. The convergence to non-rational directions $\nu$ follows by a similar argument as in Step 2, while the convergence for non-rational centers $x$ follows by a similar perturbation argument as in~\eqref{eq:firstbound} and~\eqref{eq:otherbound}.

\smallskip
\textbf{Step 5} Finally, we show the full continuity of the map $\nu\mapsto g_{\rm hom}(\w,a,b,\nu)$. The goal is to show that the limit $g_{\rm hom}$ does not depend on the orientation of the cube $Q^{\nu}_{t}$ inside the hyperplane $H^{\nu}$. This implies that we can equivalently use a orthogonal matrix-valued function $\nu\mapsto \widetilde{O}_{\nu}$ with $\widetilde{O}_{\nu}(-e_d)=\nu$ that is discontinuous at $\nu=e_d$ and by repeating Step 2 we obtain another version of $g_{\rm hom}(\w,a,b,\nu)$ that is continuous on $\mathbb{S}^{d-1}\setminus\{e_d\}$ and coincides with $g_{\rm hom}(\w,a,b,\nu)$. As we already know from Step 2 that $\nu\mapsto g_{\rm hom}(\w,a,b,\nu)$ is continuous on $\mathbb{S}^{d-1}\setminus\{-e_d\}$, this yields the continuity at all directions $\nu$.

So let us fix another orthogonal matrix $\widetilde{O}_{\nu}$ such that $\widetilde{O}_{\nu}(-e_d)=\nu$ that generates another family of cubes $\widetilde{Q}^{\nu}_t(x)$ and a possibly different function $\widetilde{g}_{\rm hom}(\w,a,b,\nu)$ satisfying \eqref{eq:existence-limit}. Given $k\in\N$, we define the family of cubes $\widetilde{\mathcal{Q}}_k=\{\widetilde{Q}^{\nu}_k(kz):\,z\in\Z^{d-1}\}$ and for $R\gg 1$ the set
\begin{equation*}
Q_{R,k}=\overline{\bigcup\{Q\in\widetilde{\mathcal{Q}}_k:\,Q\subset\subset Q^{\nu}_{Rk}\}}.	
\end{equation*}
Repeating the argument for subadditivity used in the proof of Lemma \ref{lem:subadditive-process} we deduce that
\begin{equation*}
	\m_1(\w)(u_0^{a,b,\nu},Q^{\nu}_{Rk})\leq \sum_{Q\in\widetilde{\mathcal{Q}}_k:\,Q\subset\subset Q^{\nu}_{Rk}}\m_1(\w)(u_0^{a,b,\nu},Q)+C\underbrace{\mathcal{H}^{d-1}(H^{\nu}\cap Q^{\nu}_{Rk}\setminus Q_{R,k})}_{\leq C(Rk)^{d-2}k}.
\end{equation*}
Dividing the above estimate by $(Rk)^{d-1}$ and noting that in the above sum we have at most $R^{d-1}$ terms, letting $k\to +\infty$ it follows that
\begin{equation*}
	g_{\rm hom}(\w,a,b,\nu)\leq \widetilde{g}_{\rm hom}(\w,a,b,\nu)+CR^{-1}.
\end{equation*}
Since $R$ can be made arbitrarily large, we conclude that $g_{\rm hom}\leq \widetilde{g}_{\rm hom}$. Reversing the roles of the cubes, also the reverse inequality holds true, which concludes the proof of continuity of $g_{\rm hom}$.
\end{proof}

\subsection{Proof of the lower bound}
In this section we establish the $\Gamma$-liminf inequality for Theorem \ref{thm:main}, which is formulated explicitly in Proposition \ref{prop.liminf} below. First we recall the fundamental estimate proven in \cite[Lemma 4.4]{AmBrI}.
\begin{lemma}[Fundamental estimate]\label{l.fundamental_est}
Let $K\subset\subset B\subset\subset A$ be open subsets of $D$ and $g(\w)\in \A_c$. Then for some $\Lambda>0$ and all $u,v\in BV(D;\M)$ there exists $w\in BV(D;\M)$ satisfying $w=u$ on $K$, $w=v$ on $A\setminus B$, $w(x)\in\{u(x),v(x)\}$ a.e. in $A$ and 
\begin{equation*}
	E_{\e}(\w)(w,A)\leq E_{\e}(\w)(u,B)+E_{\e}(\w)(v,A\setminus \overline{K})+\Lambda c|B\setminus\overline{K}\cap\{u\neq v\}|.
\end{equation*}
\end{lemma}
\begin{remark}\label{r.constant}
As shown in the proof in \cite{AmBrI}, the constant $\Lambda$ can be taken as the Lipschitz constant of a cut-off function between the sets $K$ and $D\setminus B$ and therefore it can be bounded by $\Lambda\leq 2\,\dist(K,\partial B)$.
\end{remark}
\begin{proposition}\label{prop.liminf}
Let $\w\in\Omega'$ with $\Omega'$ the set of full probability given by Proposition \ref{prop:existence-limit}. Then for any $D\in\A$ and any sequence $(u_{\e})\subset L^1(D;\R^m)$ such that $u_{\e}\to u$ in $L^1(D;\R^m)$ we have
\begin{equation*}
	\liminf_{\e\to 0}E_{\e}(\w)(u_{\e},D)\geq \begin{cases}
		\displaystyle\int_{D\cap S_u}g_{\rm hom}(\w,u^+,u^-,\nu_u)\,\mathrm{d}\mathcal{H}^{d-1} &\text{if}\ u\in BV(D;\M),\\
		+\infty &\text{otherwise in}\ L^1(D;\R^m).
	\end{cases}
\end{equation*}

\end{proposition} 
\begin{remark}\label{r.justopen}
	It will be clear from the proof that Proposition \ref{prop.liminf} also holds when $D$ is an arbitrary open set. We need the Lipschitz regularity in the proof of the upper bound.
\end{remark}
\begin{proof}
It suffices to consider the case when $\liminf_{\e}E_{\e}(\w)(u_{\e},D)<+\infty$. Passing to a subsequence, we can further assume that the $\liminf$ is actually a limit and hence
\begin{equation}\label{eq:energybounded}
	\lim_{\e\to 0}E_{\e}(\w)(u_{\e},D)=C<+\infty.
\end{equation}
In particular, $(u_\e)\subset BV(D;\M)$ and since $g(\omega)\in\A_c$ we deduce that $\sup_\e\|u_\e\|_{BV(D)}<+\infty$, hence $u\in BV(D;\M)$ by the compactness theorem~\cite[Theorem 3.23]{AFP}. We will use a classical blow-up argument. To this end, consider the ($\w$-dependent) Radon-measures defined by their action on a Borel set $B\subset D$ via
\begin{equation*}
	\lambda_{\e}(B)=\int_{B\cap S_{u_{\e}}}g(\w,\tfrac{x}{\e},u_{\e}^+,u_{\e}^-,\nu_{u_{\e}})\,\mathrm{d}\mathcal{H}^{d-1}.
\end{equation*}
Since $g$ is non-negative, the bound \eqref{eq:energybounded} implies that $|\lambda_{\e}|(D)$ is equi-bounded, so that up to a subsequence (not relabeled) we can assume that $\lambda_{\e}\overset{*}{\rightharpoonup}\lambda$ for some non-negative finite Radon measure $\lambda$ on $D$. Since $D$ is open, the weak$^*$-convergence together with~\cite[Proposition 1.62]{AFP} implies that
\begin{equation*}
	\lim_{\e\to 0}E_{\e}(\w)(u_{\e},D)=\lim_{\e\to 0}\lambda_{\e}(D)\geq\lambda(D)\geq\int_{D\cap S_u}\frac{{\rm d}\lambda_{\rm ac}}{\mathrm{d}\mathcal{H}^{d-1}\LL S_u}(x)\,\mathrm{d}\mathcal{H}^{d-1},
\end{equation*}
where in the last inequality we used the Besicovitch differentiation theorem (see \cite[Theorem 1.153 and Remark 1.154 (ii)]{FoLe}) to decompose $\lambda=\lambda_{\rm ac}+\lambda_{\rm s}$, where $\lambda_{\rm ac},\lambda_{\rm s}$ are non-negative finite Radon measure such that $\lambda_{\rm ac}$ is absolutely continuous with respect to the finite Radon measure $\mathcal{H}^{d-1}\LL S_u$ and  $\lambda_s\perp\mathcal{H}^{d-1}\LL S_u$. Thus it suffices to show that 
\begin{equation*}
	\frac{{\rm d}\lambda_{\rm ac}}{\mathrm{d}\mathcal{H}^{d-1}\LL S_u}(x)\geq g_{\rm hom}(\w,u^+(x),u^-(x),\nu_u(x))\qquad\text{ for }\mathcal{H}^{d-1}\LL S_u\text{-a.e. }x\in D.
\end{equation*}
Still due to the Besicovitch differentiation theorem we know that for $\mathcal{H}^{d-1}\LL S_u$-a.e. $x\in D$ we have
\begin{equation}\label{eq:besicovitch}
	\frac{{\rm d}\lambda_{\rm ac}}{\mathrm{d}\mathcal{H}^{d-1}\LL S_u}(x)=\lim_{\rho\to 0}\frac{\lambda(x+\rho C)}{\mathcal{H}^{d-1}((x+\rho C)\cap S_u)}
\end{equation}
for every bounded, convex, open set containing the origin. We now fix $x_0\in S_u$ satisfying~\eqref{eq:besicovitch} and let $\nu=\nu_u(x_0)$, $a=u^+(x_0)$ and $b=u^-(x_0)$. By well-known properties of $BV$-functions (see \cite[Theorems 2.83 and 3.78]{AFP}), we can further assume that $x_0$ satisfies
\begin{equation}\label{eq:approx-tangent}
	\lim_{\rho\to 0}\frac{\mathcal{H}^{d-1}(Q^{\nu}_{\rho}(x_0)\cap S_u)}{\rho^{d-1}}=1.
\end{equation} 
Moreover, since $\lambda$ is a finite Radon measure, we have $\lambda(\partial Q_\rho^\nu(x_0))=0$ except for countably many $\rho$.
Hence, taking $C=Q^{\nu}$, combining~\eqref{eq:besicovitch}-\eqref{eq:approx-tangent} with~\cite[Proposition 1.62]{AFP} yields that
\begin{equation*}
	\frac{{\rm d}\lambda_{\rm ac}}{\mathrm{d}\mathcal{H}^{d-1}\LL S_u}(x_0)=\lim_{\rho\to 0}\frac{\lambda(Q_{\rho}^{\nu}(x_0))}{\rho^{d-1}}=\lim_{\rho\to 0}\frac{\lambda(\overline{Q}_{\rho}^{\nu}(x_0))}{\rho^{d-1}}\geq \liminf_{\rho\to 0}\liminf_{\e\to 0}\frac{\lambda_{\e}(Q_{\rho}^{\nu}(x_0))}{\rho^{d-1}},
\end{equation*}
%where we used that up to countably many $\rho$ we have $\lambda(\partial Q_{\rho}^{\nu}(x_0))=0$ in order to apply Portmanteau's  (see~\cite[Proposition 2.62]{AFP}). 
Consequently, it suffices to show that
\begin{equation}\label{eq:localized_liminf}
\liminf_{\rho\to 0}\liminf_{\e\to 0}\frac{1}{\rho^{d-1}}\int_{Q_{\rho}^{\nu}(x_0)\cap S_{u_{\e}}}g(\w,\tfrac{x}{\e},u_{\e}^+,u_{\e}^-,\nu_{u_{\e}})\,\mathrm{d}\mathcal{H}^{d-1}	\geq g_{\rm hom}(\w,a,b,\nu).
\end{equation}
To this end, we will modify $u_{\e}$ to be the function $u_{x_0}^{a,b,\nu}$ near $\partial Q_{\rho}^{\nu}(x_0)$, which then allows us to compare the energy of this new function with the stochastic process defining $g_{\rm hom}$. 
%In order to control the energetic error, we use the fundamental estimate of Lemma \ref{l.fundamental_est}. 
%To reduce notation, in what follows we drop $x_0$ from our notation for cubes (but remember that all cubes will be centered at $x_0$). 
In order to control the energetic error, we let $\delta\in (0,1/2)$ and apply the fundamental estimate Lemma~\ref{l.fundamental_est} with the sets $K=Q^{\nu}_{(1-2\delta)\rho}(x_0)$, $B=Q^{\nu}_{(1-\delta)\rho}(x_0)$ and $A=Q_{\rho}^{\nu}(x_0)$, and the functions $u=u_{\e}$ and $v=u_{x_0}^{a,b,\nu}$. With the bound in Remark \ref{r.constant} we infer that there exists a function $\tilde{u}_{\e}\in BV(D;\M)$ with $\tilde{u}_\e=u_{x_0}^{a,b,\nu}$ in a neighborhood of $\partial Q_{\rho}^{\nu}(x_0)$ and such that
\begin{align*}
E_{\e}(\w)(\tilde{u}_{\e},A)&\leq E_{\e}(\w)(u_{\e},B)	+ E_{\e}(\w)(u_{x_0}^{a,b,\nu},A\setminus \overline{K})+\frac{C}{\delta\rho}|A\setminus \overline{K}\cap\{u_{\e}\neq u_{x_0}^{a,b,\nu}\}|
\\
&\leq E_{\e}(\w)(u_{\e},Q_{\rho}^{\nu}(x_0))+C \mathcal{H}^{d-1}(Q_{\rho}^{\nu}\setminus Q_{(1-2\delta)\rho}^{\nu}\cap H^{\nu})+\frac{C}{\delta\rho}\int_{Q_{\rho}^{\nu}(x_0)}|u_{\e}-u_{x_0}^{a,b,\nu}|\dx,
\end{align*}
where we used that $g(\w)\in \A_c$ and that $\M$ is finite. Since $\tilde{u}_\e \in\Adm(u_{x_0}^{a,b,\nu},Q_\rho^\nu(x_0))$, the left-hand side term can be bounded from below by $\m_\e(\omega)(u_{x_0}^{a,b,\nu},Q_\rho^\nu(x_0))$. Dividing the whole inequality by $\rho^{d-1}$, from Proposition~\ref{prop:existence-limit} and Remark~\ref{rem:rescaling} we deduce that
%By the change of variables $x\mapsto y/\e$, the left-hand side term can be bounded from below by $\e^{d-1}\m_1(\w)(u_{x_0/\e}^{a,b,\nu},Q_{\rho/\e}^{\nu}(x_0/\e))$. Dividing the whole inequality by $\rho^{d-1}$, from Lemma \ref{prop:existence-limit} we deduce that {\BBB [[Refer to Remark~\ref{rem:rescaling} ]]}
\begin{align*}
	g_{\rm hom}(\w,a,b,\nu)\leq \liminf_{\e\to 0}\frac{1}{\rho^{d-1}}E_{\e}(\w)(u,Q_{\rho}^{\nu}(x_0))+C (1-(1-2\delta)^{d-1})+\frac{C}{\delta\rho^{d}}\int_{Q_{\rho}^{\nu}(x_0)}|u-u_{x_0}^{a,b,\nu}|\dx.
\end{align*}
By our choice of $a=u^+(x_0)$, $b=u^-(x_0)$ and $\nu=\nu_u(x_0)$, the last averaged integral vanishes when $\rho\to 0$ by the definition of approximate jump point (see~\eqref{def:jump-point}). The claimed estimate \eqref{eq:localized_liminf} then follows by the arbitrariness of $\delta\in (0,1/2)$.
\end{proof}

\subsection{Proof of the upper bound}
In this section we prove the limsup-inequality for Theorem~\ref{thm:main}. It is stated in Proposition~\ref{prop:limsup} below. We shall need the following elementary geometric lemma.
\begin{lemma}\label{l.closeto_d-2}
Let $H^{\nu_1}(x_1)$ and $H^{\nu_2}(x_2)$ be two hyperplanes that are not parallel, i.e., $\nu_1$ and $\nu_2$ are linearly independent and let $x\in\R^d$ be a point such that $\dist(x,H^{\nu_1}(x_1))\leq \rho$ and $\dist(x,H^{\nu_2}(x_2))\leq\rho$. Then
\begin{equation*}
	\dist(x,H^{\nu_1}(x_1)\cap H^{\nu_2}(x_2))\leq c(\nu_1,\nu_2)\rho.
\end{equation*}
with $c(\nu_1,\nu_2)\defas \sqrt{1+\frac{4}{1-(\nu_1\cdot\nu_2)^2}}$.
\end{lemma}
\begin{proof}
It is well-known that for two non-parallel hyperplanes the set $H^{\nu}(x_1)\cap H^{\nu_2}(x_2)$ is a $d-2$-dimensional affine space. Up to a translation, we can assume that the intersection contains the origin and hence we can assume that $x_1=x_2=0$. Then, by orthogonal projection we have that
\begin{equation}\label{eq:closeness}
\max\{|x\cdot\nu_1|,|x\cdot\nu_2|\}\leq\rho.
\end{equation}
Consider the vector $\widetilde{\nu}_2$ constructed by a step of the Gram-Schmidt algorithm, i.e., 
\begin{equation*}
	\widetilde{\nu}_2=\frac{\nu_2-(\nu_2\cdot\nu_1) \nu_1}{|\nu_2-( \nu_2\cdot\nu_1) \nu_1|}=\frac{\nu_2-(\nu_2\cdot\nu_1) \nu_1}{\sqrt{1-(\nu_1\cdot\nu_2)^2}}.
\end{equation*}
Since $\nu_1,\widetilde{\nu}_2$ is an orthonormal basis for the orthogonal complement $(H^{\nu_1}\cap H^{\nu_2})^{\perp}$, we deduce from \eqref{eq:closeness} that
\begin{equation*}
\dist(x,H^{\nu_1}\cap H^{\nu_2})^2\leq ( x\cdot\nu_1)^2+( x\cdot\widetilde{\nu}_2)^2\leq \rho^2+\frac{4\rho^2}{1-(\nu_1\cdot\nu_2)^2}.
\end{equation*}
Taking square-roots yields the claim.
\end{proof}
\begin{proposition}\label{prop:limsup}
Let $\Omega'\subset\Omega$ be the set of full probability as in Proposition~\ref{prop:existence-limit}. For every $\omega\in\Omega'$ and every $u\in BV(D;\M)$ there exists a sequence $(u_\e)\subset BV(D;\M)$ (depending on $\omega$) such that $u_\e\to u$ in $L^1 (D;\R^m)$ and satisfying
	\begin{equation}\label{est:limsup}
	\limsup_{\e\to 0} E_\e(\omega)(u,D)\leq \int_{D\cap S_u}g_{\rm hom}(\omega, u^+, u^- ,\nu_u)\dHd.
	\end{equation}
\end{proposition}
\begin{proof}
Let $u\in BV(D;\M)$ and let $D'\in\mathscr{A}$ be such that $D\wcont D'$. 
Since $D$ has Lipschitz boundary, by \cite[Lemma 2.7]{BCG} we can extend $u$ to $D'$ in such a way that $u\in BV(D';\M)$ and $|Du|(\partial  D)=0$. 
Given $\omega\in\Omega'$, below we construct a sequence $(u_\e)\subset L^1(D')$ such that $u_\e\to u$ in $L^1(D')$ and 
	\begin{equation}\label{est:limsup-extension}
	\limsup_{\e\to 0} E_\e(\omega)(u,D)\leq \int_{D'\cap S_u}g_{\rm hom}(\omega, u^+, u^- ,\nu_u)\dHd.
	\end{equation}
If~\eqref{est:limsup-extension} holds true, using that $\Hd(S_u\cap\partial D)\leq C|Du|(\partial D)=0$, 
%by applying Proposition~\ref{prop.liminf} {\BBB with the open set $D'\setminus\overline{D}$} 
we deduce~\eqref{est:limsup} from~\eqref{est:limsup-extension} by letting $D'\searrow D$.
%	\begin{align*}
%	\limsup_{\e\to 0} E_\e(\omega)(u_\e,D) &\leq\limsup_{\e\to 0}E_\e(\omega)(u_\e,D')-\liminf_{\e\to 0} E_\e(\omega)(u_\e,D'\setminus\overline{D})\\
%	&\leq\int_{D'\cap S_u}g_{\rm hom}(\omega, u^+, u^- ,\nu_u)\dHd-\int_{(D'\setminus\overline{D})\cap S_u}g_{\rm hom}(\omega, u^+, u^- ,\nu_u)\dHd\\
%	&=\int_{D\cap S_u}g_{\rm hom}(\omega, u^+, u^- ,\nu_u)\dHd,
%	\end{align*}
%where in the first inequality we also used that $E_\e(\omega)(u_\e,\cdot)$ is increasing as a set function. Thus, \eqref{est:limsup} follows from~\eqref{est:limsup-extension}.
Eventually, since $g_{\rm hom}(\omega,\cdot,\cdot,\cdot)$ is continuous, it suffices to prove~\eqref{est:limsup-extension} for $u\in BV(D';\M)$ with polyhedral jumpset $S_u$. The general case follows thanks to the density result~\cite[Theorem 2.1 and Corollary 2.4]{BCG} and a standard lower semicontinuity argument (see \cite[Remark 1.29]{GCB}). Following the definition of polyhedral jumpset given in~\cite{BCG}, below we assume that there exist finitely many $(d-1)$-dimensional simplices $T_1,\ldots, T_N$ such that up to a $\mathcal{H}^{d-1}$-negligible set $D'\cap S_u=D'\cap\bigcup_{k=1}^N T_k$. 
%Then the triplet $(u^+,u^-,\nu)$ is piecewise constant and we write $(u^+,u^-,\nu_u)_{|T_i\cap D'}=(a_i,b_i,\nu_i)$ with $a_i,b_i\in\M$, $a_i\neq b_i$.

\smallskip
\textbf{Step 1.} To illustrate ideas, we first treat the case $N=1$. In this case we have $D'\cap S_u=D'\cap H^\nu(x_0)$ with $x_0\in T_1$ and $\nu=\nu_u$. We can thus assume that $u=u_{x_0}^{a,b,\nu}$ on $D'$ with $a,b\in\M$, $a\neq b$. 
%Setting $(a,b,\nu)\defas (a_1,b_1,\nu_1)$ and fixing a point $x_0\in T_1$ we have $u=u_{x_0}^{a,b,\nu}$ in $D'$ and $S_u\cap D'=H^\nu(x_0)\cap D'$. 
%for some $x_0\in\R^d$ and $\nu=\nu_u\in\Sd$ (up to an $\Hd$ negligible set). We can thus assume that $u=u_{x_0}^{a,b,\nu}$ on $D'$ with $a,b\in\M$, $a\neq b$. 
Let now $\rho>0$ with $\sqrt{d}\rho<\dist(D;\partial D')$ be arbitrary and let
	\begin{equation}\label{def:Irho}
	\ZZ_\rho\defas\Big\{z\in\Z^{d-1}\colon Q_\rho^\nu\big(x_0+ O_\nu(\rho z,0)\big)\cap D\neq\emptyset\Big\}.
	\end{equation}
To shorten notation set $x_\rho^z\defas x_0+O_\nu(\rho z,0)$ for every $z\in\ZZ_\rho$.
Since $(x_{\rho}^z-x_0)\cdot\nu=0$, we clearly have $u_{x_0}^{a,b,\nu}=u_{x_{\rho}^{z}}^{a,b,\nu}$. In view of~\eqref{eq:limit-rescaled} there exists $\e_\rho>0$ sufficiently small such that for every $\e\in(0,\e_\rho)$ and every $z\in\ZZ_\rho$ we find a function $u_{\e,\rho}^{z}\in\Adm(u_{x_0}^{a,b,\nu},Q_\rho^\nu(x_{\rho}^z))$ with
	\begin{equation}\label{est:limsup-almost-optimal}
	\frac{1}{\rho^{d-1}}E_\e(\omega)\big(u_{\e,\rho}^{z},Q_\rho^\nu(x_{\rho}^z)\big)\leq g_{\rm hom}(\w,a,b,\nu)+\rho.
	\end{equation}
Note that $\bigcup_{z\in\ZZ_\rho}Q_\rho^\nu(x_\rho^z)\subset\{\dist(\cdot,D)<\sqrt{d}\rho\}\subset\subset D'$. We define $u_{\e,\rho}\in BV(D';\M)$ setting $u_{\e,\rho}(x)=u_{\e,\rho}^z(x)$ if $x\in Q_\rho^\nu(x_\rho^z)$ for some $z\in\ZZ_\rho$ and by extending $u_{\e,\rho}$ via $u=u_{x_0}^{a,b,\nu}$ on $D'\sm\bigcup_{z\in\ZZ_\rho}Q_\rho^\nu(x_\rho^z)$. By construction of the cubes $Q_{\rho}^\nu(x_{\rho}^z)$ we deduce that
%In this way, we have $u_{\e,\rho}=u$ outside $\{\dist(x,D\cap H^\nu(x_0))\leq\sqrt{d}\rho\}$, hence
	\begin{equation}\label{est:limsup-L1norm}
	\|u_{\e,\rho}-u\|_{L^1(D';\R^m)}\leq C\big|\big\{x\in D'\colon\dist(x,H^\nu(x_0))< \sqrt{d}\rho\big\}\big|
	\leq C {\rm diam}(D')^{d-1}\rho.
	\end{equation}
for every $\e\in(0,\e_\rho)$.

We next estimate the energy of $u_{\e,\rho}$. Since $\Hd\big(S_{u_{\e,\rho}}\cap\partial Q_\rho^\nu(x_\rho^z)\big)=0$ for every $z\in\ZZ_\rho$ due to the boundary conditions satisfied by every $u_{\e,\rho}^z$, in view of~\eqref{est:limsup-almost-optimal} we get
	\begin{equation}\label{est:limsup-energy1}
	E_\e(\omega)(u_{\e,\rho},D) \leq\sum_{z\in\ZZ_\rho} E_\e(\omega)(u_{\e,\rho},Q_\rho^\nu(x_\rho^z))\\
	\leq \#\ZZ_\rho\big(\rho^{d-1}g_{\rm hom}(\omega,a,b,\nu)+\rho^{d}\big)
	\end{equation}
for every $\e<\e_{\rho}$. The cardinality of $\ZZ_\rho$ can be estimated via
	\begin{equation}\label{est:cardinality-Ir}
	\begin{split}
	\rho^{d-1}\#\ZZ_\rho &=\sum_{z\in\I_\rho}\Hd\big(Q_\rho^\nu(x_\rho^z)\cap H^\nu(x_0)\big)\leq \mathcal{H}^{d-1}(D'\cap H^{\nu}(x_0))=\mathcal{H}^{d-1}(D'\cap S_u).
	\end{split}
	\end{equation}
Combining \eqref{est:limsup-energy1} and~\eqref{est:cardinality-Ir}, we infer that
	\begin{equation}\label{est:limsup-energy2}
	\limsup_{\rho\to 0}\limsup_{\e\to 0}E_\e(\omega)(u_{\e,\rho},D)\leq \int_{D'\cap S_u}g_{\rm hom}(\omega,a,b,\nu)\dH^{d-1}.
	\end{equation}
Thanks to~\eqref{est:limsup-L1norm} and~\eqref{est:limsup-energy2} a diagonal argument provides us with a sequence $(u_\e)=(u_{\e,\rho(\e)})$ with $\|u_\e-u\|_{L^1(D';\R^m)}\to 0$ and such that \eqref{est:limsup-extension} holds.

\smallskip
\textbf{Step 2.} We now treat the case $N\geq2$. 
%Then each simplex $T_k$ is contained in a hyperplane $H^{\nu_k}(x_k)$ with $x_k\in T_k$. In general we might have $H^{\nu_k}(x_k)=H^{\nu_m}(x_m)$ for $k\neq m$.
Upon relabeling the simplices $T_1,\ldots, T_N$ we can choose $x_k\in T_k$ and $\nu_k\in\Sd$ for $k\in\{1,\ldots,M\}$, $M\leq N$, such that each simplex is contained in a hyperplane $H^{\nu_k}(x_k)$ and we have $H^{\nu_k}(x_k)\neq H^{\nu_m}(x_m)$ for $k\neq m$. By the assumption on $u$ this implies that
	\begin{equation}\label{covering:jumpset-hyperplane}
	\Hd\bigg(D'\cap S_u\sm\bigcup_{k=1}^MH^{\nu_k}(x_k)\bigg)=0.
	\end{equation}	 
Let now $\rho>0$ be sufficiently small such that $\sqrt{d}\rho<\dist(D,\partial D')$ and such that $\rho$ is smaller than the minimal distance between two parallel hyperplanes $H^{\nu_k}(x_k)$, $H^{\nu_m}(x_m)$ with $\nu_k=\nu_m$. For such $\rho$, up to varying $x_k\in T_k$ we can further assume that
\begin{equation}\label{eq:nojumponcubeboundary}
	\mathcal{H}^{d-1}(\partial Q_{\rho}^{\nu_k}(\underbrace{x_k+O_{\nu_k}(\rho z,0)}_{=:x_{\rho}^{k,z}})\cap D'\cap S_u)=0\qquad\text{ for all }z\in\Z^{d-1}.
\end{equation}

\smallskip
%Similar as in Step 1, we set $x_\rho^{k,z}\defas x_k+O_{\nu_k}(\rho z,0)$ for every $z\in\Z^{d-1}$ and $k\in\{1,\ldots,M\}$ and we define
%	\begin{equation*}
%	\ZZ_\rho^k\defas\big\{z\in\Z^{d-1}\colon \Hd\big(Q_\rho^{\nu_k}(x_\rho^{k,z})\cap S_u\cap D\big)>0\big\}.
%	\end{equation*}
%By assumption, the cubes $Q_\rho^{\nu_k}(x_\rho^{k,z})$, $k\in\{1,\ldots, M\}$, $z\in\ZZ_\rho^k$ cover $S_u\cap D$ up to an $\Hd$-negligible set but in contrast to Step 1 the cover is not disjoint, since cubes with different orientation might overlap. To avoid this, we introduce the smaller families of cubes labeled by the set $\ocirc{\ZZ_\rho^k}$ with
%	\begin{equation}\label{def:Zrho-interior}
%	\ocirc{\ZZ_\rho^k}\defas\big\{z\in\ZZ_\rho^k\colon \overline Q_\rho^{\nu_k}(x_\rho^{k,z})\cap \overline Q_\rho^{\nu_m}(x_\rho^{m,z'})=\emptyset\;\text{ for every $m\neq k$ and every}\ z'\in\ZZ_\rho^m\big\}
%	\end{equation}
%Let now $k\in\{1,\ldots,M\}$ be fixed and suppose that $z\in\ocirc{\ZZ_\rho^k}$. Since the cubes $Q_\rho^{\nu_m}(x_\rho^{m,z'})$ with $m\neq k$ and $z'\in\ZZ_\rho^m$ cover $S_u\sm H^{\nu_k}(x_k)$ up to an $\Hd$-negligible set, the condition on the intersection in~\eqref{def:Zrho-interior} ensures that $\Hd\big(S_u\cap Q_\rho^{\nu_k}(x_\rho^{k,z})\cap H^{\nu_m}(x_m)\big)=0$ for every $m\neq k$. In view of~\eqref{covering:jumpset-hyperplane} this in turn implies that
%	\begin{equation*}
%	\Hd\big(Q_\rho^{\nu_k}(x_{\rho}^{k,z})\cap S_u\sm H^{\nu_k}(x_k)\big)=0.
%	\end{equation*}
%We can thus assume that 
In view of~\eqref{covering:jumpset-hyperplane} the family of cubes $\{Q_\rho^{\nu_k}(x_\rho^{k,z})\}_{k,z}$ covers $S_u\cap D'$ up to an $\Hd$-negligible set, but in contrast to Step 1 the cover is not disjoint, since cubes with different orientation might overlap. To avoid this, for any $k=1,\ldots,M$ we set
	\begin{equation}\label{def:Zrho-1}
	\ZZ_\rho^1 \defas\big\{z\in\Z^{d-1}\colon Q_\rho^\nu\big(x_{\rho}^{1,z})\cap D\neq\emptyset\big\}
	\end{equation}
and for $k\in\{2,\ldots,M\}$ we iteratively define $\ZZ_\rho^k$ via 
	\begin{equation}\label{def:Zrho-k}
	\ZZ_\rho^k \defas\bigg\{z\in\Z^{d-1}\colon Q_\rho^\nu\big(x_{\rho}^{k,z})\cap D\neq\emptyset\; \text{ and }\; Q_\rho^{\nu_k}(x_{\rho}^{k,z})\cap \bigcup_{m=1}^{k-1}\bigcup_{z'\in \ZZ_\rho^m}Q_\rho^{\nu_m}(x_\rho^{m,z'})=\emptyset\bigg\}.
	\end{equation} 
We further subdivide the classes $\ZZ_\rho^k$ according to the traces of $u$ on $Q_\rho^{\nu_k}(x_\rho^{k,z})\cap H^{\nu_k}(x_k)$. 
To this end, it is convenient to write $\M=\{a_1,\ldots, a_K\}$ with $K=\#\M$ and to set $A_i\defas\{u(x)=a_i\}$ for $i\in\{1,\ldots,K\}$. For every $k\in\{1,\ldots,M\}$ and every pair $(i,j)\in\{1,\ldots,K\}^2$ with $i\neq j$ we let
%More precisely, for every $k\in\{1,\ldots,M\}$ and every pair $(a,b)\in\M\times\M$ with $a\neq b$ we set
	\begin{equation}\label{def:Zrho-ijk}
	\ZZ_{\rho}^k(i,j)\defas\big\{z\in\ZZ_\rho^k\colon \big|Q_\rho^{\nu_k,+}(x_\rho^{k,z})\sm A_i\big|=\big|Q_\rho^{\nu_k,-}(x_\rho^{k,z})\sm A_j\big|=0\big\},
	\end{equation}
where $Q_\rho^{\nu_k,\pm}(x_\rho^{k,z})$ is defined as in~\eqref{def:Qpm}. In this way, every $z\in\ZZ_\rho^k$ belongs to at most one family $\ZZ_\rho^k(i,j)$, $(i,j)\in\{1,\ldots,K\}^2$ $(i\neq j)$ and by the definition of approximate jump points in~\eqref{def:jump-point} we have that for $\Hd$-a.e.\ $x\in Q_\rho^{\nu_k}(x_\rho^{k,z})\cap H^{\nu_k}(x_k)$ the triplet $(u^+(x),u^-(x),\nu_u(x))$ coincides with $(a_i,a_j,\nu_k)$. In particular, we have
	\begin{equation}\label{eq:jumpset-Qrkz}
	Q_\rho^{\nu_k}(x_\rho^{k,z})\cap H^{\nu_k}(x_k)=Q_\rho^{\nu_k}(x_\rho^{k,z})\cap\rb A_i\cap \rb A_j
	\end{equation}
up to an $\Hd$-negligible set.
Let us also write 
	\begin{equation*}
	\ZZ_\rho^{k,{\rm co}}\defas\ZZ_\rho^k\sm\bigcup_{\begin{smallmatrix}(i,j)\in\{1,\ldots,K\}^2\\i\neq j\end{smallmatrix}}\ZZ_\rho^k(i,j)
	\end{equation*}
for the complement of all $\ZZ_\rho^k(i,j)$ in $\ZZ_\rho^k$.
For every $k\in\{1,\ldots,M\}$ and $z\in\ZZ_\rho^k$ we define a function $u_{\e,\rho}^{k,z}\in BV(Q_\rho^{\nu_k}(x_\rho^{k,z});\M)$ by distinguishing the following two exhaustive cases. If $z\in\ZZ_\rho^{k,{\rm co}}$, we set $u_{\e,\rho}^{k,z}\defas u$ on $Q_\rho^{\nu_k}(x_\rho^{k,z})$. If $z\not\in\ZZ_\rho^{k,{\rm co}}$, then $z\in\ZZ_\rho^k(i,j)$ for a unique pair $(i,j)\in\{1,\ldots,K\}^2$ with $i\neq j$ and for $\e$ sufficiently small (depending on $\rho$) we choose $u_{\e,\rho}^{k,z}\in\Adm(u_{x_k}^{a_i,a_j,\nu_k},Q_\rho^{\nu_k}(x_\rho^{k,z}))$ with
	\begin{equation}\label{est:limsup-minimiser}
	\frac{1}{\rho^{d-1}}E_\e(\omega)\big(u_{\e,\rho}^{k,z},Q_\rho^{\nu_k}(x_\rho^{k,z})\big)\leq g_{\rm hom}(\omega,a_i,a_j,\nu_k)+\rho. 
	\end{equation}
We eventually define $u_{\e,\rho}\in BV(D';\M)$ by setting
	\begin{equation}\label{def:recovery}
	u_{\e,\rho}(x)\defas
	\begin{cases}
	u_{\e,\rho}^{k,z}(x) &\text{if $x\in Q_\rho^{\nu_k}(x_\rho^{k,z})$ for some $k\in\{1,\ldots,M\}$ and $z\in\ZZ_\rho^k$,}\\
	u(x) &\text{otherwise in $D'$.}
	\end{cases}
	\end{equation}
As in~\eqref{est:limsup-L1norm} we find that $u_{\e,\rho}$ satisfies the $L^1$-estimate $\|u_{\e,\rho}-u\|_{L^1(D';\R^m)}\leq CN{\rm diam}(D')^{d-1}\rho$ and it remains to estimate its energy. By the definition of the sets $\mathcal{Z}_{\rho}^k(i,j)$ we know that on every possible cube $Q_{\rho}^{\nu_k}(x_{\rho}^{k,z})$ it holds that $u_{\e,\rho}=u$ in a neighborhood of $\partial Q_{\rho}^{\nu_k}(x_{\rho}^{k,z})$. Thus equation \eqref{eq:nojumponcubeboundary} ensures that $\Hd\big(S_{u_{\e,\rho}}\cap \partial Q_\rho^{\nu_k}(x_\rho^{k,z})\big)=0$ . Combined with \eqref{covering:jumpset-hyperplane} this implies that
	\begin{equation}\label{est:limsup-general1}
	E_\e(\omega)(u_{\e,\rho},D) \leq\sum_{k=1}^M\sum_{z\in\ZZ_\rho^k}E_\e(\omega)\big(u_{\e,\rho}^{k,z},Q_\rho^{\nu_k}(x_\rho^{k,z})\big)+c\sum_{k=2}^M\sum_{z\in\Z^{d-1}\sm\ZZ_\rho^k}\Hd\big(Q_\rho^{\nu_k}(x_\rho^{k,z})\cap D\cap S_u\big).
	\end{equation}
We start estimating the last term in~\eqref{est:limsup-general1}. For any $k\in\{2,\ldots,M\}$ and for any $z\in\Z^{d-1}\sm\ZZ_\rho^k$ with $Q_\rho^{\nu_k}(x_\rho^{k,z})\cap D\neq\emptyset$ there exist $m\in\{1,\ldots,k-1\}$ and $z'\in\ZZ_\rho^m$ such that $Q_\rho^{\nu_k}(x_\rho^{k,z})\cap Q_\rho^{\nu_m}(x_\rho^{m,z'})\neq\emptyset$. Since $\rho$ is smaller than the minimal distance between two parallel hyperplanes, we know that $H^{\nu_k}(x_k)$ and $H^{\nu_m}(x_m)$ are not parallel and we can apply Lemma \ref{l.closeto_d-2} . In fact, choosing $\hat x\in Q_\rho^{\nu_k}(x_\rho^{k,z})\cap Q_\rho^{\nu_m}(x_\rho^{m,z'})$ we have for every $x\in Q_\rho^{\nu_k}(x_\rho^{k,z})$
	\begin{equation}\label{est:dist-hyperplane}
	\dist(x,H^{\nu_m}(x_m))\leq |x-\hat{x}|+|\hat{x}-x_{\rho}^{m,z'}|\leq 2\sqrt{d}\rho\quad\text{and}\quad\dist(x,H^{\nu_k}(x_k))\leq |x-x_{\rho}^{k,z}|\leq\tfrac{\sqrt{d}}{2}\rho.
	\end{equation}
Since moreover $\sqrt{d}\rho<\dist(D,\partial D')$, we deduce from Lemma~\ref{l.closeto_d-2} that 
	\begin{equation}\label{inclusion:near-hyperplane}
	Q_\rho^{\nu_k}(x_\rho^{k,z})\subset\big\{x\in D'\colon\dist(x;H^{\nu_k}(x_k)\cap H^{\nu_m}(x_m)\leq 2\sqrt{d}c(\nu_k,\nu_m)\rho\big\}.
	\end{equation}
%
%implies that $\dist\big(Q_\rho^{\nu_k};H^{\nu_k}(x_k)\cap H^{\nu_m}(x_m))\leq C\rho$ for some constant $C>0$ depending on $\nu_k$ and $\nu_m$. 
The collection of all such cubes is thus contained in a $\rho$-neighborhood of $H^{\nu_k}(x_k)\cap\bigcup_{m=1}^{k-1} H^{\nu_m}(x_m)\cap D'$. In addition, we have $\Hd\big(Q_\rho^{\nu_k}(x_\rho^{k,z})\cap S_u\big)\leq C\rho^{d-1}= C\frac{|Q_\rho^{\nu_k}(x_\rho^{k,z})|}{\rho}$, since $S_u$ is polyhedral. Thus, setting $\hat{c}\defas2\sqrt{d}\max\{c(\nu_k,\nu_m)\colon\nu_k,\nu_m\text{ linearly independent}\}<+\infty$ and recalling that the  collection of cubes is pairwise disjoint we obtain
	\begin{equation}\label{est:limsup-rho-term1}
	\begin{split}
	\sum_{z\in\Z^{d-1}\sm\ZZ_\rho^k}\Hd\big(Q_\rho^{\nu_k}(x_\rho^{k,z})\cap D\cap S_u\big)&\leq \frac{C}{\rho}\sum_{m=1}^{k-1}\big|\{x\in D'\colon\dist(x,H^{\nu_k}(x_k)\cap H^{\nu_m}(x_m))<\hat{c}\rho\}\big|\leq C\rho.
	\end{split}
	\end{equation}
It remains to estimate the first term on the right-hand side of~\eqref{est:limsup-general1}. For every $k\in\{1,\ldots,M\}$ we deduce from~\eqref{est:limsup-minimiser} that
	\begin{equation}\label{est:limsup-general2}
	\begin{split}
	\sum_{z\in\ZZ_\rho^k}E_\e(\omega)\big(u_{\e,\rho}^{k,z},Q_\rho^{\nu_k}(x_\rho^{k,z})\big) &\leq \dsum{(i,j)\in\{1,\ldots,K\}^2}{i\neq j}\#\ZZ_\rho^k(i,j)\big(\rho^{d-1}g_{\rm hom}(\omega,a_i,a_j,\nu_k)+\rho^d\big)\\
	&\hspace*{1em} +c\sum_{z\in\ZZ_\rho^{k,{\rm co}}}\Hd\big(S_u\cap Q_\rho^{\nu_k}(x_\rho^{k,z})\big).
	\end{split}
	\end{equation}
We claim that also the last term in~\eqref{est:limsup-general2} vanishes with $\rho$. This can be seen as follows: let $z\in\ZZ_\rho^{k,{\rm co}}$; if $|Q_\rho^{\nu_k}(x_\rho^{k,z})\sm A_i|=0$ for some $i\in\{1,\ldots,K\}$, then $\Hd\big(S_u\cap Q_\rho^{\nu_k}(x_\rho^{k,z})\big)=0$, hence $z$ does not contribute to~\eqref{est:limsup-general2}. Otherwise we know that either $|Q_\rho^{\nu_k,+}(x_\rho^{k,z})\sm A_i|>0$ for all $i\in\{1,\ldots,K\}$ or $|Q_\rho^{\nu_k,-}(x_\rho^{k,z})\sm A_i|>0$ for all $i\in\{1,\ldots,K\}$. We assume without loss of generality that the first situation holds and we choose $i\in\{1,\ldots K\}$ with $|Q_\rho^{\nu_k,+}(x_\rho^{k,z})\cap A_i|>0$. An application of the relative isoperimetric inequality then yields
	\begin{equation*}
	0<\min\Big\{\big|Q_\rho^{\nu_k,+}(x_\rho^{k,z})\cap A_i\big|,\big|Q_\rho^{\nu_k,+}(x_\rho^{k,z})\sm A_i\big| \Big\}^\frac{d-1}{d}\leq C\Hd\big(Q_\rho^{\nu_k,+}(x_\rho^{k,z})\cap\rb A_i\big).
	\end{equation*}
In view of~\eqref{covering:jumpset-hyperplane} this in turn implies that $\Hd\big(Q_\rho^{\nu_k,+}(x_\rho^{k,z})\cap H^{\nu_m}(x_m)\big)>0$ for some hyperplane $H^{\nu_m}(x_m)$ not parallel to $H^\nu_k(x_k)$. Arguing as in~\eqref{est:dist-hyperplane} and using Lemma~\ref{l.closeto_d-2} we see that $Q_\rho^{\nu_k}(x_\rho^{k,z})$ satisfies~\eqref{inclusion:near-hyperplane}. Thus, as in~\eqref{est:limsup-rho-term1} we conclude that
	\begin{equation}\label{est:limsup-rho-term2}
	\sum_{z\in\ZZ_\rho^{k,{\rm co}}}\Hd\big(S_u\cap Q_\rho^{\nu_k}(x_\rho^{k,z})\big)\leq C\rho.
	\end{equation}
Eventually, to estimate the first term in~\eqref{est:limsup-general2} we observe that
%Similar to~\eqref{est:cardinality-Ir} the cardinality of $\ZZ_\rho^k(i,j)$ can be estimated via
	\begin{equation}\label{est:cardinality-Zrkab}
	\begin{split}
	\rho^{d-1}\#\ZZ_\rho^k(i,j) &\leq \sum_{z\in\ZZ_\rho^k(i,j)}\Hd\big(Q_\rho^{\nu_k}(x_\rho^{k,z})\cap H^{\nu_k}(x_k)\big)\\
	&=\sum_{z\in\ZZ_\rho^k(i,j)}\Hd\big(Q_\rho^{\nu_k}(x_\rho^{k,z})\cap H^{\nu_k}(x_k)\cap\rb A_i\cap\rb A_j\big)\\
	&\leq\Hd\big(D'\cap H^{\nu_k}(x_k)\cap\rb A_i\cap\rb A_j\big),
	\end{split}
	\end{equation}
where to obtain the second equality we used~\eqref{eq:jumpset-Qrkz}, while the last estimate follows since $Q_\rho^{\nu_k}(x_\rho^{k,z})\subset D'$. Eventually, summing up~\eqref{est:cardinality-Zrkab} in~\eqref{est:limsup-general2} and combining it with~\eqref{est:limsup-general1} and~\eqref{est:limsup-rho-term1}-\eqref{est:limsup-rho-term2} we infer
	\begin{equation}\label{est:limsup-final}
	\begin{split}
	E_\e(\omega)(u_{\e,\rho},D) &\leq\sum_{k=1}^M\dsum{(i,j)\in\{1,\ldots,K\}^2}{i\neq j}\int_{D'\cap\rb A_i\cap\rb A_j\cap H^{\nu_k}(x_k)}g_{\rm hom}(\omega,a_i,a_j,\nu_k)\dHd+C\rho\\
	&=\sum_{k=1}^M\int_{D'\cap S_u\cap H^{\nu_k}(x_k)}g_{\rm hom}(\omega,u^+,u^-,\nu_u)\dHd+C\rho,
	\end{split}
	\end{equation}
where the last equality follows from~\cite[Proposition 5.9]{AFP} together with~\eqref{eq:jumpset-Qrkz}. Gathering~\eqref{covering:jumpset-hyperplane} and~\eqref{est:limsup-final} we conclude again by a diagonal argument.
\end{proof}

\subsection{Proof with boundary conditions}
\begin{proof}[Proof of Theorem \ref{thm:boundary}]
The proof of Theorem~\ref{thm:boundary} will be carried out in several steps.

\smallskip
\textbf{Step 1.} We construct a suitable neighborhood of $\Gamma$ and a suitable function attaining the boundary datum in this neighborhood. This preliminary step will be essential both for the lower and the upper bound.
We start constructing the neighborhood of $\Gamma$. Let $\eta>0$ be arbitrary. Since we assume that the relative boundary of $\Gamma$ is $\mathcal{H}^{d-1}$-negligible, also the spherical Hausdorff-measure of $\Gamma_{\rm rel}$ vanishes. Hence, for any $\delta>0$ there exist countably many balls $
B_{r_{i}}(z_i)$ with $r_i<\delta$ that cover $\partial_{\rm rel}\Gamma$ and such that
\begin{equation}\label{eq:goodcovering}
	\sum_{i=1}^{\infty}r_i^{d-1}\leq\eta.
\end{equation} 
In particular, choosing $\delta\leq\eta$ we can ensure that $r_i<\eta$. The Lipschitz regularity and compactness of $\partial D$ allow us to further assume that
\begin{equation}\label{eq:Lipschitzconcentration}
	\mathcal{H}^{d-1}(\partial D\cap B_{r_i}(z_i))<Cr_i^{d-1},\qquad \mathcal{H}^{d-1}(\partial D\cap\partial B_{r_i}(z_i))=0.
\end{equation}
for some constant $C$ only depending on $\partial D$.
For $x\in\Gamma$, we further choose balls $B_{r_x}(x)$ such that $r_x<\eta$ and $B_{r_x}(x)\cap\partial_{\rm rel}\Gamma=\emptyset$, which is possible since $\Gamma$ is relatively open in $\partial D$. Moreover, again due to the fact that $\mathcal{H}^{d-1}\LL\partial D$ is a finite measure we can ensure that $\mathcal{H}^{d-1}(\partial D\cap \partial B_{r_x}(x))=0$. Since $\overline{\Gamma}=\Gamma\cup\partial_{\rm rel}\Gamma$, we can use the balls constructed above and the compactness of $\overline{\Gamma}$ to find a finite subcover $\{B_{r_i}(z_i)\}_{i=1}^N\cup\{B_{r_{x_j}}(x_j)\}_{j=1}^N$ that still covers $\overline{\Gamma}$. Let $U_{\eta}$ be the union of these balls, so that $U_{\eta}$ is an open neighborhood of $\overline{\Gamma}$ and a set of finite perimeter. Moreover, since the boundary of all balls has negigible intersection with $\partial D$, it holds that
\begin{equation}\label{eq:partialZ0}
	\mathcal{H}^{d-1}(\partial U_{\eta}\cap\partial D)=0.
\end{equation}
Next, given $u\in BV(D;\M)$, the Lipschitz regularity of $D$ allows us to use \cite[Lemma 2.7]{BCG} to find an extension $\widetilde{u}\in BV_{\rm loc}(D;\M)$ of $u$ such that $|D\widetilde{u}|( \partial D)=0$. We finally define the $BV_{\rm loc}(\R^d;\M)$-function 
\begin{equation}\label{eq:defv_eta}
	v_{\eta}(u)=u_0\mathds{1}_{U_{\eta}}+\widetilde{u}(1-\mathds{1}_{U_{\eta}}).
\end{equation}
Since $U_\eta$ is a neighborhood of $\Gamma$, we have that $v_{\eta}(u)=u_0$ in a neighborhood of $\Gamma$. Moreover, by \cite[formula (3.81)]{AFP} and \eqref{eq:partialZ0} we have
\begin{equation}\label{eq:nomassonpartialD}
	|Dv_{\eta}(u)|(\partial D)\leq |Du_0|(\partial D)+|D\widetilde{u}|(\partial D)+C\mathcal{H}^{d-1}(\partial U_{\eta}\cap\partial D)=0.
\end{equation}
With this set $U_{\eta}$ and the construction of the function $v_{\eta}(u)$ at hand, we can now turn to the actual proof of Theorem \ref{thm:boundary}.

\smallskip
\textbf{Step 2.} We prove the $\liminf$-inequality. To this end it suffices to consider a sequence $u_{\e}\to u$ in $L^1(D;\R^m)$ such that $\liminf_{\e}E_{\e}(\w)(u_{\e},D)<+\infty$. Passing to a subsequence, we can further assume that the $\liminf$ is actually a limit and therefore $u_{\e}\in BV_{\Gamma,u_0}(D;\M)$. Since $g(\w)\in\A_c$, we deduce that $u\in BV(D;\M)$. We extend $u_{\e}$ to $U_{\eta}\setminus D$ by $u_0$ and call this extension $v_{\e}$. Then $v_{\e}$ converges in $L^1(D\cup U_{\eta};\R^m)$ to the function $u\mathds{1}_D+u_0\mathds{1}_{U_{\eta}\setminus D}$. Its energy can be bounded by
\begin{equation}\label{eq:split}
	E_{\e}(\w)(v_{\e},D\cup U_\eta)\leq E_{\e,\Gamma,u_0}(\w)(u_{\e},D)+E_{\e}(\w)(v_{\e},U_\eta)\leq E_{\e,\Gamma,u_0}(\w)(u_{\e},D)+C|Dv_{\e}|(U_\eta\setminus D).
\end{equation}
In order to estimate the term $|Dv_{\e}|(U_\eta\setminus D)$, we write $U_{\eta}\setminus D=(U_{\eta}\cap\partial D) \cup (U_{\eta}\setminus\overline{D})$. On the open set $U_{\eta}\setminus\overline{D}$, we have $v_{\e}=u_0$, so that 
\begin{equation*}
	|Dv_{\e}|(U_{\eta}\setminus D)= |Du_0|(U_{\eta}\setminus\overline{D})+|Dv_{\e}|(U_{\eta}\cap\partial D).
\end{equation*}
We decompose the set $U_{\eta}\cap\partial D$ into the intersection of $\partial D$ with the balls $B_{r_{x_j}}(x_j)$ and with the balls $B_{r_i}(z_i)$ (cf. Step 1).
Since the first family of balls does not intersect $\Gamma_{\rm rel}$, we know that $B_{r_{x_j}}(x_j)\cap\partial D\subset\Gamma$ and on $\Gamma$ the functions $u_{\e}$ and $u_0$ have the same trace. Therefore $|Dv_{\e}|$ does not have mass on those balls. For the remaining balls $B_{r_i}(z_i)$ we use the bounds in \eqref{eq:Lipschitzconcentration} and \eqref{eq:goodcovering} to find that
\begin{equation*}
	|Dv_{\e}|\left(\bigcup_{i=1}^N B_{r_i}(z_i)\cap \partial D\right)\leq C\sum_{i=1}^N\mathcal{H}^{d-1}(B_{r_i}(z_i)\cap\partial D)\leq C\sum_{i=1}^Nr_i^{d-1}\leq C\eta.
\end{equation*}
Hence we obtain the bound
\begin{equation*}
	|Dv_{\e}|(U_{\eta}\setminus D)\leq |Du_0|(U_{\eta}\setminus\overline{D})+C\eta.
\end{equation*}
Inserting this $\e$-independent estimate into \eqref{eq:split} and using Proposition \ref{prop.liminf} on the open set $D\cup U_{\eta}$ (cf. Remark \ref{r.justopen}), we find as $\e\to 0$ the inequality
\begin{align*}
	\int_{D\cap S_u}g_{\rm hom}(\w,u^+,u^-,\nu_u)\dHd+\int_{\Gamma}g_{\rm hom}(\w,u|_{\partial D},u_0|_{\partial D},\nu_{\Gamma})\dHd\leq &\liminf_{\e\to 0}E_{\e,\Gamma,u_0}(\w)(u_{\e},D)
	\\
	& \quad+C|Du_0|(U_{\eta}\setminus\overline{D})+C\eta,
\end{align*}
where on the left-hand side we neglected the non-negative energy contribution on $U_{\eta}\setminus (D\cup\Gamma)$. Since $U_{\eta}\subset\Gamma+B_{2\eta}(0)$ by the bound on the radii, the terms in the second line vanish as $\eta\to 0$ and we obtain the claimed lower bound.

\smallskip
\textbf{Step 3.} As an intermediate step towards the proof of the $\limsup$-inequality we construct for any $u\in BV(D;\M)$ a sequence $(u_n)\subset BV(D;\M)$ with $u_n\to u$ in $L^1(D;\R^m)$ and $u_n=u_0$ in a neighbourhood of $\Gamma$ and such that
\begin{equation}\label{eq:density}
	E_{{\rm hom},\Gamma,u_0}(\w)(u,D)=\lim_{n\to +\infty}E_{{\rm hom},\Gamma,u_0}(\w)(u_n,D)=\lim_{n\to+\infty}\int_{D\cap S_{u_n}}g_{\rm hom}(\omega,u_n^+,u_n^-,\nu_{u_n})\dHd.
\end{equation}
To this end, for $u\in BV(D;\M)$ let its extension $\widetilde{u}$ and the function $v_{\eta}(u)$ be as in Step 1. Moreover, let $D'$ be an open set with Lipschitz boundary such that $D\subset\subset D'$. Since $|Dv_{\eta}(u)|(\partial D)=0$ (see \eqref{eq:nomassonpartialD}), we can use \cite[Lemma B.1]{BCR} to find a sequence of sets $D_k\subset\subset D$ with finite perimeter such that
\begin{align*}
	u_{k,\eta}:=\mathds{1}_{D_k}\widetilde{u}+(1-\mathds{1}_{D_k})v_{\eta}(u)&\to u_{\eta}:= \mathds{1}_{D}\widetilde{u}+(1-\mathds{1}_{D})v_{\eta}(u)\quad \text{ in }L^1(D';\R^m),
	\\
	\mathcal{H}^{d-1}(S_{u_{k,\eta}}\cap D')&\to \mathcal{H}^{d-1}(S_{u_{\eta}}\cap D')
\end{align*}
as $k\to +\infty$. Due to \cite[Theorem 3.1]{R18ACV} these two limits imply that
\begin{equation*}
	\lim_{k\to +\infty}E_{\rm hom}(\w)(u_{k,\eta},D')=E_{\rm hom}(\w)(u_{\eta},D').
\end{equation*}
Since $E_{\rm hom}(\w)(\cdot,D)$ is $L^1(D)$-lower semicontinuous, the integrand $g_{\rm hom}(\w)$ is $BV$-elliptic and therefore $E_{\rm hom}(\w)(\cdot,U)$ is also $L^1(U)$-lower semicontinuous on all bounded, open sets $U\subset\R^d$ (cf. \cite[Theorem 2.1]{AmBrII}). Hence we deduce that
\begin{align*}
	\limsup_{k\to +\infty}E_{\rm hom}(\w)(u_{k,\eta},D)&\leq\lim_{k\to +\infty}E_{\rm hom}(\w)(u_{k,\eta},D')-\liminf_{k\to+\infty}E_{\rm hom}(\w)(u_{k,\eta},D'\setminus \overline{D})
	\\
	&\leq E_{\rm hom}(\w)(u_{\eta},D')-E_{\rm hom}(\w)(u_{\eta},D'\setminus \overline{D})=E_{\rm hom}(\w)(u_{\eta},\overline{D})
	\\
	&=\int_{S_u\cap D}g_{\rm hom}(\w,u^+,u^-,\nu_u)\dHd+\int_{\partial D}g_{\rm hom}(\w,\widetilde{u}^+,v_{\eta}(u)^-,\nu_{\partial D})\dHd.
\end{align*}
Let us rewrite the last integral. Since $\widetilde{u}$ extends $u$, we know that $\widetilde{u}^+=u|_{\partial D}$ on $\partial D$. From the definition of $v_{\eta}(u)$ in \eqref{eq:defv_eta} and the fact that $|D\widetilde{u}|(\partial D)=0$, we deduce that
\begin{align*}
	v_{\eta}(u)^-=\widetilde{u}^-=\widetilde{u}^+\qquad\mathcal{H}^{d-1}\text{-a.e. on }\partial D\setminus \overline{U_{\eta}},
	\\
	v_{\eta}(u)^-=u_0^-\qquad\mathcal{H}^{d-1}\text{-a.e. on }\partial D\cap U_{\eta}.
\end{align*}
Taking into account also \eqref{eq:partialZ0} and that $|Du_0|(\partial D)=0$, this implies that
\begin{equation*}
	\int_{\partial D}g_{\rm hom}(\w,\widetilde{u}^+,v_{\eta}(u)^-,\nu_{\partial D})\dHd=\int_{\partial D\cap U_{\eta}}g_{\rm hom}(\w,u|_{\partial D},u_0|_{\partial D},\nu_{\partial D})\dHd.
\end{equation*}
The sets $U_{\eta}$ shrink to $\overline{\Gamma}$ when $\eta$ tends to zero as it is the covering of $\overline{\Gamma}$ made of balls with radius less than $\eta$ and each containing a point from $\overline{\Gamma}$. Hence, letting $\eta\to 0$ and using that $\mathcal{H}^{d-1}(\partial_{\rm rel}\Gamma)=0$, the dominated convergence theorem yields that 
\begin{align*}
	&\limsup_{\eta\to 0}\limsup_{k\to +\infty}\int_{S_{u_{k,\eta}}\cap D}g_{\rm hom}(\w,u_{k,\eta}^+,u_{k,\eta}^-,\nu_{u_{k,\eta}})\dHd
	\\
	&\leq\int_{S_u\cap D}g_{\rm hom}(\w,u^+,u^-,\nu_u)\dHd+\int_{\Gamma}g_{\rm hom}(\w,u|_{\partial D},u_0|_{\partial D},\nu_{\Gamma})\dHd
	=E_{{\rm hom},\Gamma,u_0}(\w)(u,D).
\end{align*}
Now recall that $U_{\eta}$ is an open neighborhood of $\Gamma$ and that $u_{k,\eta}=v_{\eta}(u)$ in a neighborhood of $\partial D$. Recalling the definition of $v_{\eta}(u)$ in \eqref{eq:defv_eta}, it follows that $u_{k,\eta}=u_0$ in an open neighborhood of $\Gamma$. Hence the left-hand side of the above inequality equals $E_{{\rm hom},\Gamma,u_0}(\w)(u_{k,\eta},D)$. Moreover, each sequence $u_{k,\eta}$ converges to $u$ in $L^1(D;\R^m)$. Using a diagonal argument along $\eta\to 0$ and $k\to+\infty$, we thus find a sequences $(u_n)\subset BV_{\Gamma,u_0}(D;\M)$ and $\eta_n\to 0$ such that $u_n=v_{\eta_n}(u)$ in a neighborhood of $\partial D$, $u_n\to u$ in $L^1(D;\R^m)$ and such that~\eqref{eq:density} is satisfied.

\smallskip
\textbf{Step 4.} We establish the $\limsup$-inequality for each $u_n$, which then yields the claim due to the lower semicontinuity of the abstract $\Gamma$-$\limsup$. So fix $u_n$ and choose open sets $D_1\subset\subset D_2\subset\subset D$ such that $u_n=v_{\eta_n}(u)$ on $D\setminus \overline{D_1}$. Consider further a sequence $u_{\e,n}$ such that $u_{\e,n}\to u_n$ in $L^1(D;\R^m)$ and 
\begin{equation*}
	\limsup_{\e\to 0}E_{\e}(\w)(u_{\e,n},D)\leq E_{\rm hom}(\w)(u_n,D)=E_{{\rm hom},\Gamma,u_0}(\w)(u_n,D)
\end{equation*}
provided by Proposition~\ref{prop:limsup}.
Let us apply Lemma \ref{l.fundamental_est} with the sets $D_1,D_2$ and $D$ and the functions $u_{\e,n}$ and $v_{\eta_n}(u)$. We obtain a new sequence $w_{\e}\in BV(D;\M)$ such that $w_{\e,n}=u_{\e,n}$ on $D_1$, $w_{\e,n}=v_{\eta_n}(u)$ on $D\setminus D_2$, $w_{\e,n}(x)\in\{u_{\e,n}(x),v_{\eta_n}(u)(x)\}$ a.e. in $D$ and 
\begin{align*}
	E_{\e}(\w)(w_{\e,n},D)&\leq E_{\e}(\w)(u_{\e,n},D_2)+E_{\e}(\w)(v_{\eta_n}(D\setminus\overline{D_1})+C\dist(D_1,\partial D_2)\left|D_2\setminus\overline{D_1}\cap\{u_{\e,n}\neq v_{\eta_n}(u)\}\right|
	\\
	&\leq E_{\e}(\w)(u_{\e,n},D)+C |Dv_{\eta_n}|(D\setminus \overline{D_1}))+C\dist(D_1,\partial D_2)\int_{D_2\setminus \overline{D_1}}|u_{\e,n}-v_{\eta_n}(u)|\dx.
\end{align*} 
Since $u_n=v_{\eta_n}(u)$ on $D\setminus\overline{D_1}$, the convergence $u_{\e,n}\to u_n$ implies that $w_{\e,n}\to u_n$ in $L^1(D;\R^m)$ as well. Moreover, again by the choice of $D_1$ the last integral in the above estimate vanishes as $\e\to 0$. Finally, by construction we have that $w_{\e,n}\in BV_{\Gamma,u_0}(D;\M)$ and therefore the choice of $u_{\e,n}$ yields that
\begin{equation*}
	\limsup_{\e\to 0}E_{\e,\Gamma,u_0}(\w)(w_{\e,n},D)\leq E_{{\rm hom},\Gamma,u_0}(\w)(u_n,D)+C |Dv_{\eta_n}|(D\setminus \overline{D_1}).
\end{equation*}
Now letting $D_1\uparrow D$ the last term vanishes since $|Dv_{\eta_n}(u)|$ is a finite measure. This yields the $\limsup$-inequality for $u_n$ and we conclude the proof.
\end{proof}

\subsection*{Acknowledgments}
A.\ Bach has received funding from the European Union's Horizon research and innovation programme under the Marie Sk\l odowska-Curie grant agreement No 101065771.

\appendix

\section{Measurability of the stochastic process}
\begin{proposition}\label{prop:measurable}
	Let $g:\Omega\times\R^d\times\mathcal{M}\times\mathcal{M}\times\mathbb{S}^{d-1}\to [0,+\infty)$ be as in Definition \ref{defadmissible}, $A\subset\R^d$ be a bounded, open set with Lipschitz boundary and $v\in BV_{\rm loc}(\R^d;\M)$. Then the map $\w\mapsto\m_1(\w)(v,A)$ defined in \eqref{def:min-problem} is $\mathcal{F}$-measurable.
\end{proposition}
\begin{proof}
Without loss of generality we can assume that $\M=\M':=\{e_1,\ldots,e_k\}\subset\R^k$ for $k=\#\M$. Indeed, otherwise write $\M=\{m_1,\ldots,m_k\}$ and consider the bijection $T:\M\to\M'$ defined by $T(m_i)=e_i$ and the new integrand $\widetilde{g}(\w,x,e_i,e_j,\nu):=g(\w,x,T^{-1}(e_i),T^{-1}(e_j),\nu)$, which satisfies
\begin{equation*}
	\m_1(\w,v,A)=\inf\left\{\int_{A\cap S_u}\widetilde{g}(\w,x,u^+,u^-,\nu_u)\,\mathrm{d}\mathcal{H}^{d-1}:\,u\in BV(A;\mathcal{M}'),\,u=T\circ v\text{ near }\partial A\right\}.
\end{equation*}
Note that $\widetilde{g}$ satisfies all the properties in Definition \ref{defadmissible} and that $T\circ v\in BV_{\rm loc}(\R^d;\M')$. Hence from now on we assume that $\M=\M'$. Consider the set $(\M-\M)\setminus\{0\}=\{e_i-e_j:\,i\neq j,\,1\leq i,j\leq k\}$. Note that each element in $(\M-\M)\setminus\{0\}$ has exactly one representation as $e_i-e_j$ with $i\neq j$. Hence we can unambiguously define a map $h:\Omega\times\R^d\times\R^{k}\times\mathbb{S}^{d-1}\to [0,+\infty)$ setting
\begin{equation*}
	h(\w,x,u,\nu)=\begin{cases}
		g(\w,x,e_i,e_j,\nu) &\mbox{if $u=e_i-e_j\in(\M-\M)\setminus\{0\}$,}
		\\
		0 &\mbox{otherwise.}
	\end{cases}
\end{equation*}
Since $(\M-\M)\setminus\{0\}$ is finite, the joint measurability of $g$ implies that $h$ is $\mathcal{F}\times\mathcal{B}(\R^d)\times\mathcal{B}(\R^{k})\times \mathcal{B}(\mathbb{S}^{d-1})$-measurable. Moreover, the symmetry condition \eqref{cond:Ac} implies that $h(\w,x,-u,-\nu)=h(\w,x,u,\nu)$. Finally, since $\M$ is a finite set the uniform boundedness of $g$ implies that $|h(\w,x,u,\nu)|\leq C |u|$ for a finite constant $C>0$. Hence we can apply \cite[Lemma A.9]{CDMSZ22} to deduce that for every $R>0$ the map
\begin{equation*}
	(\w,u)\to \int_{A\cap S_u}h(\w,x,u^+-u^-,\nu_u)\,\mathrm{d}\mathcal{H}^{d-1}
\end{equation*}
is $\mathcal{F}\otimes \mathcal{B}(BV_{R,A}^k)$-measurable, where $\mathcal{B}(BV_{R,A}^k)$ denotes the Borel $\sigma$-algebra on the compact metric space $BV_{R,A}^m$ defined as
\begin{equation*}
	BV_{R,A}^m:=\{u\in BV(A;\R^m):\,\|u\|_{L^1(A)}\leq R,\,|Du|(A)\leq R\}
\end{equation*}
equipped with the metric $d_{R,A}^m(u,v)=\|u-v\|_{L^1(A)}+d_{R,A}^{k\times d}(Du,Dv)$, where $d_{R,A}^{k\times d}$ is a distance that metrizes the weak$^*$-topology on the set of $\R^{k\times d}$-valued Radon measures with total variation bounded by $R$. For $u\in BV(A,\M)$ we have by construction that
\begin{equation*}
	\int_{A\cap S_u}h(\w,x,u^+-u^-,\nu_u)\,\mathrm{d}\mathcal{H}^{d-1}=\int_{A\cap S_u}g(\w,x,u^+,u^-,\nu_u)\,\mathrm{d}\mathcal{H}^{d-1}.
\end{equation*}
Moreover, the set $BV_{R,A}(\M):=\{u\in BV_{R,A}^m:\,u\in BV(A,\M)\}$ is a closed subspace of $BV_{R,A}(\M)$, so in particular also compact. Moreover, the map
\begin{equation*}
	(\w,u)\to \int_{A\cap S_u}g(\w,x,u^+,u^-,\nu_u)\,\mathrm{d}\mathcal{H}^{d-1}
\end{equation*}
is $\mathcal{F}\times \mathcal{B}(BV_{R,A}(\M))$-measurable. Next we rewrite the infimum problem in a more tractable way. Take a sequence of sets $(A_j)_{j\in\N}\subset\mathscr{A}$ such that $A_j\subset\subset A_{j+1}$ and $A=\bigcup_jA_j$. Then for $R$ large enough (depending only on $c$ and $|Dv|(A)$) and all $\w\in\Omega$ such that $g(\w,\cdot)\in\A_c$ we have
\begin{align}\label{eq:limitrep}
	\m_1(\w)(v,A)&=\lim_{j\to +\infty}\inf\left\{\int_{A\cap S_u}g(\w,x,u^+,u^-,\nu_u)\,\mathrm{d}\mathcal{H}^{d-1}:\,u\in BV(A;\mathcal{M}),\,u=v\text{ on }A\setminus A_j\right\}\nonumber
	\\
	&=\lim_{j\to +\infty}\inf\left\{\int_{A\cap S_u}g(\w,x,u^+,u^-,\nu_u)\,\mathrm{d}\mathcal{H}^{d-1}:\,u\in BV_{R,A}(\mathcal{M}),\,u=v\text{ on }A\setminus A_j\right\}.
\end{align}
Here we used that for $u\in BV(A;\M)$ and $g\in \A_c$ the energy and the total variation are comparable up to multiplicative constant, which gives a bound on on the total variation of almost minimizers by the total variation of $v$ on $A$. Finally, the constraint $u=v$ on $A\setminus A_j$ is closed in $BV_{R,A}(\M)$, so that the penalty function $P_j:BV_{R,A}(\M)\to\{0,+\infty\}$ defined by
\begin{equation*}
	P_j(u)=\begin{cases} 
		0 &\mbox{if $u=v$ on $A\setminus A_j$,}\\
		+\infty &\mbox{otherwise}
		\end{cases}
\end{equation*}
is $\mathcal{B}(BV_{R,A}(\M))$-measurable. We deduce that for every $t>0$ 
\begin{equation*}
	\left\{(\w,u)\in \Omega\times BV_{R,A}(\M):\,\int_{A\cap S_u}g(\w,x,u^+,u^-,\nu_u)\,\mathrm{d}\mathcal{H}^{d-1}+P_j(u)<t\right\}\in\F\otimes\mathcal{B}(BV_{R,A}(\M)).
\end{equation*}
Now we can use the projection theorem. Indeed, we assume that our probability space is complete and we know that $BV_{R,A}(\M)$ is a compact metric space and therefore complete and separable. Thus \cite[Theorem 1.136]{FoLe} implies that the projections of the above sub-level sets onto $\Omega$ are measurable. Said projections are given by
\begin{equation*}
	\left\{\w\in \Omega:\,\inf_{u\in BV_{R,A}(\M)}\int_{A\cap S_u}g(\w,x,u^+,u^-,\nu_u)\,\mathrm{d}\mathcal{H}^{d-1}+P_j(u)<t\right\}\in\mathcal{F},
\end{equation*} 
which in regard to the definition of $P_j$ and \eqref{eq:limitrep} proves that $\m_1(\w)(v,A)$ is almost surely given by the limit of measurable functions and the completeness of the probability space yields that $\m_1(\cdot)(v,A)$ is itself $\mathcal{F}$-measurable.
\end{proof}


\begin{thebibliography}{99}
	
\bibitem{AkKr} 
\newblock M.A. Akcoglu and U. Krengel:
\newblock Ergodic theorems for superadditive processes.
\newblock \emph{J. Reine Ang. Math.}, \textbf{323} (1981) , 53--67.
%
%\bibitem{Alex}
%K.~Alexander, 
%\newblock Approximation of subadditive functions and convergence rates in limiting-shape results,
%\newblock {\em Annals of Probab.} {\bf 1} (1997), 30-–55.
	
\bibitem{ACR}
R.~Alicandro, M.~Cicalese and M.~Ruf,
\newblock Domain formation in magnetic polymer composites: an approach via stochastic homogenization,
\newblock {\em Arch. Ration. Mech. Anal.}, {\bf 218}(2), 2015, 945--984.

\bibitem{AmBrI}
L.~Ambrosio and A.~Braides,
\newblock Functionals defined on partitions of sets of finite perimeter I: integral representation and $\Gamma$-convergence,
\newblock {\em J. Math. Pures. Appl.}, \textbf{69} (1990), 285--305.

\bibitem{AmBrII}
L.~Ambrosio and A.~Braides,
\newblock Functionals defined on partitions of sets of finite perimeter II: semicontinuity, relaxation and homogenization,
\newblock {\em J. Math. Pures. Appl.}, \textbf{69} (1990), 307--333.

\bibitem{AFP}
	L.~Ambrosio, N.~Fusco and D.~Pallara,
	\newblock {\em Functions of bounded variation and free discontinuity problems},
	\newblock Oxford Mathematical Monographs, The Clarendon Press Oxford University
	Press, New York, 2000.
	
%\bibitem{AKM}
%S.~Armstrong, T.~Kuusi and J.-C.~Mourrat,
%\newblock The additive structure of elliptic homogenization, 
%\newblock {\em Invent. math.}, {\bf 208} (2017), 999–-1154.

\bibitem{AKM_book}
S.~Armstrong, T.~Kuusi and J.-C.~Mourrat,
\newblock {\em Quantitative stochastic homogenization and large-scale regularity},
\newblock Grundlehren der mathematischen Wissenschaften 352, Springer Nature Switzerland AG, Cham, 2019.
%
%	
%\bibitem {AM}
%S.~Armstrong and J.-C.~Mourrat, Lipschitz regularity for elliptic equations with random coefficients.
%\newblock {\em Arch. Ration. Mech. Anal.}, {\bf 219} (2016), 255-–348. 	
%
%\bibitem{ArmSma}	
%S.~Armstrong and C.~K.~Smart, Quantitative stochastic homogenization of convex integral functionals,
%\newblock {\em Ann. Sci. Ec. Norm. Sup\'er.}, {\bf 49} (2016), 423--481.
%
%\bibitem{50years}
%A.~Auffinger, M.~Damron and J.~Hanson,
%\newblock {\em 50 years of first passage percolation},
%\newblock In: University Lecture Series, volume 68. American Mathematical Society, 2017.	
%	
\bibitem{BCR19}
A.~Bach, M.~Cicalese and M.~Ruf,
\newblock Random finite-difference discretizations of the Ambrosio-Tortorelli functional with optimal mesh size,
\newblock {\em Siam J. Math. Anal.}, {\bf 53} (2021), 2275--2318.
%
\bibitem{BMZ21}
A.~Bach, R.~Marziani and C.~I.~Zeppieri,
\newblock  $\Gamma$-convergence and stochastic homogenisation of singularly-perturbed elliptic functionals,
\newblock {\em Arxiv preprint 2102.09872} (2021).
%
\bibitem{BaRu22}
A.~Bach and M.~Ruf,
Fluctuation estimates for the multi-cell formula in stochastic homogenization of partitions,
\textit{Calc. Var. Partial Differential Equations}, {\bf 61} (2022), 84.
%\bibitem{BaCaRe}
%E.~N.~Barron, P.~Cardaliaguet and R.~Jensen,
%\newblock Conditional essential supremum with applications.
%\newblock {\em Appl. Math. Optim.} {\bf 48} (2003), 229--253.
%
%\bibitem{BoLuMa}
%S.~Boucheron, G.~Lugosi and P.~Massart, 
%\newblock Concentration inequalities using the entropy method,
%\newblock {\em Ann. Probab.}, {\bf 31} (2003), 1583--1614.

\bibitem{BFLM}
G.~Bouchitt\`e, I.~Fonseca, G.~Leoni and L.~Mascarenhas, 
	\newblock{A global method for relaxation in $W^{1,p}$ and $SBV_p$}, 
	\newblock \emph{Arch. Ration. Mech. Anal.}, {\bf 165} (2002), 187--242.

\bibitem{BraidesFD}
\newblock A.~Braides,
\newblock \emph{Approximation of free-discontinuity problems}
\newblock Lecture Notes in Mathematics, Springer Verlag, Berlin, 1998.	
	
\bibitem{GCB}
\newblock A.~Braides,
\newblock \emph{$\Gamma$-convergence for Beginners}.
\newblock Oxford Lecture Series in Mathematics and its Applications 22, Oxford University Press, Oxford, 2002.

\bibitem{BrCi} A.~Braides and M.~Cicalese,
\newblock Interfaces, modulated phases and textures in lattice systems.
\newblock {\em Arch. Ration. Mech. Anal.} {\bf 223} (2017), 977--1017.

	
\bibitem{BCG}
A.~Braides, S.~Conti and A.~Garroni,
\newblock Density of polyhedral partitions.
\newblock {\em Calc. Var. Partial Diff. Eq.}, \textbf{56} (2017), art. 28.	

\bibitem{BCR} {A. Braides, M. Cicalese and M. Ruf}, Continuum limit and stochastic homogenization of discrete ferromagnetic thin films. {\em Anal. PDE} {\bf 11} (2018), 499--553.

\bibitem{BP}
A.~Braides and A.~Piatnitski,
\newblock Homogenization of surface and length energies for spin systems,
\newblock {\em J. Func. Anal.}, {\bf 264} (2013), 1296--1328.

\bibitem{BP20}
A.~Braides and A.~Piatnitski, 
\newblock Homogenization of ferromagnetic energies on Poisson random sets in the plane, 
\newblock{\em Arch. Ration. Mech. Anal.}, {\bf 243} (2022), 433–458.

%\bibitem{CadlL}
%L.~A.~Caffarelli and R.~de la Llave,
%\newblock Planelike minimizers in periodic media,
%\newblock {\em Commun. Pure Appl. Math.}, {\bf 54} (2001), 1403--1441.

\bibitem{CDMSZ19}
	F.~Cagnetti, G.~Dal~Maso, L.~Scardia and C.~I.~Zeppieri,
	\newblock {Stochastic homogenisation of free discontinuity problems}.
	\newblock {\em Arch. Ration. Mech. Anal.} {\bf 233} (2019), 935--974.
	
\bibitem{CDMSZ22} 
F.~Cagnetti, G.~Dal~Maso, L.~Scardia and C.~I.~Zeppieri,
\newblock {A global method for deterministic and stochastic homogenization in BV}
\newblock {\em Ann. PDE},  {\bf 8} no. 1 (2022), Paper No. 8, 89 pp.
 
%\bibitem{DMMoI} 
%G.~Dal~Maso and L.~Modica, Nonlinear stochastic homogenization,
%\newblock {\em Ann. Mat. Pura Appl.} {\bf 144} (1986), 347-–389,.
%
%\bibitem{DMMoII} 
%G.~Dal~Maso and L.~Modica, Nonlinear stochastic homogenization and ergodic theory,
%\newblock {\em J.Reine Angew. Math.}, {\bf 368} (1986), 28–42.	

\bibitem{DG}
M.~Duerinckx and A.~Gloria,
\newblock Stochastic homogenization of nonconvex unbounded integral functionals with convex growth,
\newblock {\em Arch. Ration. Mech. Anal.}, {\bf 221} (2016), 1511--1584.
	

%\bibitem{DG1}
% M.~Duerinckx and A.~Gloria,
% \newblock Multiscale functional inequalities in probability: concentration properties,
% \newblock {\em ALEA, Lat. Am. J. Probab. Math. Stat.} {\bf 17} (2020), 133--157.
%
%\bibitem{DG2}
% M.~Duerinckx and A.~Gloria,
% \newblock Multiscale functional inequalities in probability: constructive approach,
% \newblock {\em Ann. Henri Lebesgue}, {\bf 3} (2020), 825--872.
% 
% \bibitem{FiNe}
% J.~Fischer and S.~Neukamm, 
% \newblock Optimal homogenization rates in stochastic homogenization of nonlinear uniformly elliptic equations and systems,
% \newblock Preprint, arXiv 1908.02273 (2020).

\bibitem{FoLe} 
\newblock I. Fonseca and G. Leoni,
\newblock \emph{Modern Methods in the Calculus of Variations: $L^p$ spaces},
\newblock Springer, New York, 2007.

% 
% \bibitem{TrSl}
% N.~Garc\'ia Trillos, D.~Slep\v{c}ev,
% \newblock Continuum limit of total variation on point clouds,
% \newblock {\em  Arch. Ration. Mech. Anal.}, {\bf 220} (2016), 193–-241.
% 
% \bibitem{GNO_Inv}
% A.~Gloria, S.~Neukamm and F.~Otto, Quantification of ergodicity in stochastic homogenization: optimal bounds via spectral gap on Glauber dynamics,
% \newblock {\em Invent. math.}, {\bf 199} (2015), 455-–515.
% 
% \bibitem{GO_Var}
% A.~Gloria and F.~Otto,
% \newblock An optimal variance estimate in stochastic homogenization of discrete elliptic equations,
% \newblock {\em Ann. Probab.}, {\bf 39} (2011), 779--856.
% 
% \bibitem{GO_Error}
% A.~Gloria and F.~Otto, An optimal error estimate in stochastic homogenization of discrete elliptic equations, 
% {\em Ann. Appl. Probab.}, {\bf 22} (2012), 1–-28.
% 
% \bibitem{GO_Jems}
% A.~Gloria and F.~Otto, Quantitative results on the corrector equation in stochastic homogenization. 
% \newblock {\em J. Eur. Math. Soc.}, {\bf 19} (2017), 3489--3548.
% 
% \bibitem{GO_Semi}
% A.~Gloria and F.~Otto, The corrector in stochastic homogenization: optimal rates, stochastic integrability, and fluctuations, 
% \newblock Preprint, arXiv 1510.08290 (2016).
 
  \bibitem{JKO_Hom}
 V.~V.~Jikov, S.~M.~Kozlov and O.~A.~Oleinik,
 \newblock {\em Homogenization of differential operators and integral functionals}, Springer-Verlag, Berlin, 1994.
 
% \bibitem{Koz} S.~M.~Kozlov,
% \newblock The averaging of random operators,
% \newblock {\em Mat. Sb. (N.S.)}, {\bf 109}(151) (1979), 188–-202.

 \bibitem{Marziani}
 R.~Marziani,
 \newblock $\Gamma$-convergence and stochastic homogenisation of phase-transition functionals,
 \newblock {\it Arxiv preprint 2206.13131} (2022). 
 
 \bibitem{Morfe}
 P.~S.~Morfe,
 \newblock Surface tension and $\Gamma$-convergence of van der Waals–Cahn–Hilliard phase transitions in stationary ergodic media,
 \newblock {\em  J. Stat. Phys.}, {\bf 181} (2020), 2225-–2256.
 
% \bibitem{NaSp} A.~Naddaf and T.~Spencer,
% \newblock Estimates on the variance of some homogenization problems,
% \newblock Unpublished preprint, 1998.
 
% \bibitem{PaVa} G.~C.~Papanicolaou and S.~R.~S.~Varadhan,
% \newblock {\em Boundary value problems with rapidly oscillating random coefficients,} In: Random Fields, Vol. I, II (Esztergom, 1979), volume 27 of Colloquia Mathematica Societatis J\'anos Bolyai, pages 835–873,
% \newblock North-Holland, Amsterdam, 1981.
 
% \bibitem{PSZ20}
% X.~Pellet, L.~Scardia and C.~I.~Zeppieri,
% \newblock {Stochastic homogenisation of free-discontinuity functionals in random perforated domains}
% \newblock {\em Preprint arXiv:2002.01389} (2020).
% 
% \bibitem{RoThe}
% R. Rossignol and M. Th\'eret,
% \newblock Lower large deviations and laws of large numbers for maximal flows through a box in first passage percolation,
% \newblock {\em Ann. Inst. H. Poincar\'e Probab. Statist.} {\bf 46} (2010), 1093--1131.
% 
 \bibitem {R17} M.~Ruf,
 \newblock Discrete stochastic approximations of the Mumford-Shah functional.
 \newblock {\em Ann. Inst. H. Poincar\'e Anal. Non Lin\'eaire}, \textbf{36} (2019), 887--937.
 
\bibitem{R18ACV} M.~Ruf,
\newblock On the continuity of functionals defined on partitions.
\newblock {\em Adv. Calc. Var.}, {\bf 11} (2018), 335--339.

\bibitem{RR23} M.~Ruf and T.~Ruf:
\newblock Stochastic homogenization of degenerate integral functionals and their Euler-Lagrange equations.
\newblock {\em J. Ec. Polytech. Math.}, {\bf 10} (2023), 253--303.

\bibitem{RZ23} M.~Ruf and C.~I.~Zeppieri,
\newblock Stochastic homogenization of degenerate integral functionals with linear growth.
\newblock {\em Arxiv preprint 2210.14829} (2022).
% 
% \bibitem{The}
% M.~Th\'eret,
% \newblock Upper large deviations for maximal flows through a tilted cylinder,
% \newblock {\em ESAIM: PS}, {\bf 18} (2014), 117--129.
 
 %\bibitem{Yur} V.~V.~Yurinskii,
 %\newblock Averaging of symmetric diffusion in random medium,
 %\newblock {\em Siberian Math. J.}, {\bf 27} (1986), 603--613.
 
\end{thebibliography}
\end{document}